\documentclass[reqno,11pt]{amsart}
\usepackage{amsfonts,amsmath,amssymb}
\usepackage{graphicx}

\numberwithin{equation}{section}

\def\be{{\beta}}

\def\eps{\epsilon}

\def\varep{\varepsilon}

\def\al{\alpha}

\def\R{\mathbb{R}}

\newtheorem{theorem}{Theorem}[section]
\newtheorem{lemma}[theorem]{Lemma}

\newtheorem{proposition}[theorem]{Proposition}
\newtheorem{definition}[theorem]{Definition}
\newtheorem{remark}[theorem]{Remark}

\begin{document}

\title{On the absence of "splash" singularities in the case of two-fluid interfaces}

\author{Charles Fefferman}
\address{Princeton University}
\email{cf@math.princeton.edu}

\author{Alexandru D. Ionescu}
\address{Princeton University}
\email{aionescu@math.princeton.edu}

\author{Victor Lie}
\address{Purdue University}
\email{vlie@math.purdue.edu}

\thanks{The first author is supported in part by NSF grant DMS-0901040. The second author is supported in part by a Packard
Fellowship and NSF grant DMS-1265818. The third author is supported in part by NSF grant DMS-0100932.}

\begin{abstract}
We show that "splash" singularities cannot develop in the case of locally smooth solutions of the two-fluid interfaces in two dimensions. More precisely, we show that the scenario of formation of singularities discovered by Castro--C\'{o}rdoba--Fefferman--Gancedo--G\'{o}mez-Serrano \cite{CaCoFeGaGo} in the case of the water waves system, in which the interface remains locally smooth but self-intersects in finite time, is completely prevented in the case of two-fluid interfaces with positive densities. 
\end{abstract}

\maketitle

\tableofcontents

\section{Introduction}\label{intro}

In this paper we are interested in studying the formation of singularities in the case of two-fluid interfaces in two dimensions, taking both the gravity and the surface tension into account. We will show that the scenarios of singularity formation ("splash" and "splat") discovered by Castro--C\'{o}rdoba--Fefferman--Gancedo--G\'{o}mez-Serrano in \cite{CaCoFeGaGo} in the case of the water wave system are completely prevented in the more stable case of a two-fluid interface. 

To state the equations, assume that we are working on a time interval $I$, and for any $t\in I$ we have an embedded curve (the interface) $z=z(\alpha,t)=(z_1(\alpha,t),z_2(\alpha,t))\in C^4(\mathbb{R}\times I:\mathbb{R}^2)$. We assume that the interface separates the plane into two open regions $\Omega^1(t), \,\Omega^2(t)\subseteq\mathbb{R}^2$. More precisely, we let $\Gamma_t:=\{z(\alpha,t):\alpha\in\mathbb{R}\}$ and assume that the plane $\mathbb{R}^2$ can be written as a disjoint union 
\begin{equation*}
\mathbb{R}^2=\Omega^1(t)\cup\Omega^2(t)\cup\Gamma_t.
\end{equation*}
We assume that $\Omega^1(t)$ and $\Omega^2(t)$ are open, connected, and simply-connected sets, and that $\Gamma_t=\partial\Omega^1(t)=\partial\Omega^2(t)$. Finally, we assume that there is $a\geq 1$ such that $\mathbb{R}\times(-\infty,-a]\subseteq \Omega^1(t)$ and $\mathbb{R}\times[a,\infty)\subseteq \Omega^2(t)$ for all $t\in I$.

\begin{figure}[hb]
\centering
\includegraphics[width=5in]{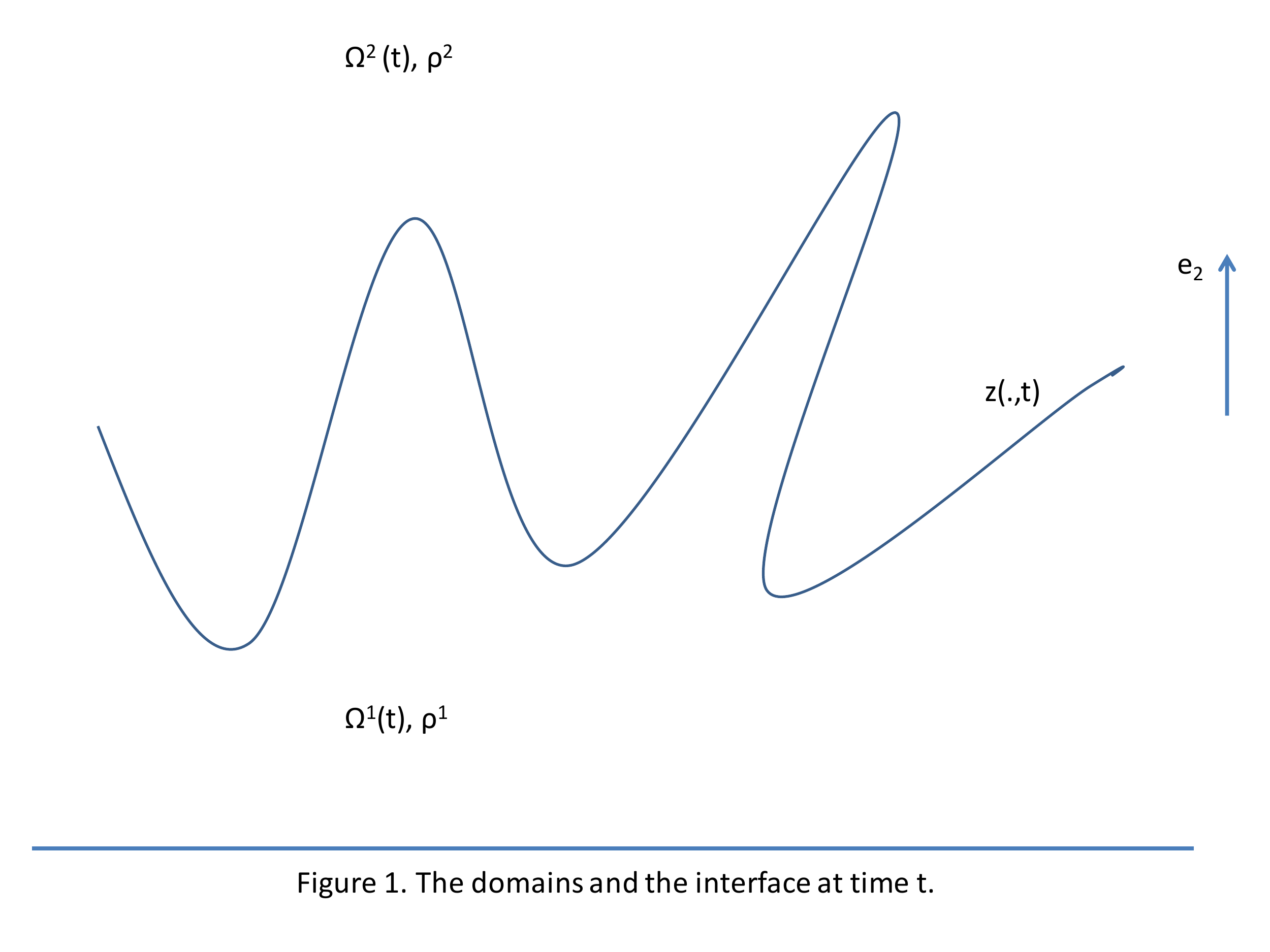}
\end{figure}

We assume that the domains $\Omega^1(t)$ and $\Omega^2(t)$ are filled with fluids with constant densities $\rho^1,\rho^2\geq 0$, and sufficiently smooth velocity fields $u^j(t)\in \cap C^3(\overline{\Omega^j(t)})$, $j\in\{1,2\}$. We assume that both fluids are incompressible in their respective domains, 
\begin{equation}\label{cv1}
\mathrm{div}\,u^j=0\qquad\text{ in }\Omega^j(t),\,j\in\{1,2\},
\end{equation}
and that they evolve according to the Euler equation
\begin{equation}\label{cv2}
\rho^{j} (\partial_t u^{j}+[u^{j}\cdot \nabla]u^{j})=-\nabla P^{j}-g\rho^{j}e_2,\qquad\text{ in }\Omega^j(t),\,j\in\{1,2\},
\end{equation}
where $g\geq 0$ is the gravity constant. We assume that the interface $\Gamma_t$ is a bounding surface (fluid particles do not cross it) which moves with the fluid. In particular, the velocity fields $u^1(t)$ and $u^2(t)$ satisfy suitable compatibility conditions at the interface. More precisely, we assume that
\begin{equation}\label{cv3}
(\partial_tz-u^j|_{\Gamma_t})\cdot(\partial_\alpha z)^\perp=0,\qquad j\in\{1,2\}.
\end{equation}
Finally, the continuity of the stress tensor at the interface gives
\begin{equation}\label{cv4}
(P_2-P_1)|_{\Gamma_t}=-\sigma K(z)=-\sigma\frac{\partial_\alpha z_1\partial_\alpha^2z_2-\partial_\alpha z_2\partial_\alpha^2z_1}{[(\partial_\alpha z_1)^2+(\partial_\alpha z_2)^2]^{3/2}},
\end{equation}
where $\sigma\geq 0$ is the surface tension coefficient.

The equations \eqref{cv1}--\eqref{cv4} are a typical set of equations that describe the evolution of two-fluid interfaces.\footnote{There is a slight imprecision in these equations, namely the function $z(,.t)$ is determined only up to a reparametrization. This is consistent, of course, with the physical interpretation of the equations. To eliminate this imprecision, there are at least two natural parametrizations one could use: Eulerian coordinates in the case when the interface is a graph above a Euclidean space (notice that the vertical direction is determined invariantly by the choice of the vector-field $e_2$ in the gravity term in \eqref{cv2}), or Lagrangian coordinates relative to one of the fluids. In this paper we work invariantly, i.e. we do not make any specific choice of coordinates on the interface.} The equations extend easily to the case of two-dimensional interfaces (i.e. three-dimensional domain $\Omega^1$ and $\Omega^2$), and also to other situations, such as fluids with finite depth. 

The equations depend on four nonnegative parameters: the densities $\rho^1,\rho^2$, the surface tension coefficient $\sigma$, and the gravity coefficient $g$. Several special cases can be considered. For example, the special case $\rho^2=0$ corresponds to the one-fluid interface problem (free surface Euler equations), while the special case $\rho^1=\rho^2$ corresponds to the Kelvin--Helmholtz problem. In the case of the one-fluid interface problem one can further specialize to irrotational flows, the so-called water waves system.

The equations \eqref{cv1}--\eqref{cv4} form a quasi-linear evolution system, and have been studied extensively in recent years. As with most evolution equations, there are at least three natural questions one can consider: (1) local regularity, (2) formation of singularities, and (3) global regularity. We discuss below some of the recent developments:

\subsubsection{Local regularity} The question of local existence of solutions in the case of "nice" data is, of course, the most basic question for any evolution system. In the case of two-fluid models this question is reasonably well understood. For the irrotational case of the water waves problem, and for 2D fluids (and hence 1D interfaces), the earliest
local existence results were obtained by Nalimov \cite{Na}, Yosihara \cite{Yo}, and Craig \cite{Cr} for initial data near equilibrium. In \cite{Wu1} and \cite{Wu2} Wu was able to obtain local-in-time existence of solutions in Sobolev spaces for the gravity water waves system in two and three dimensions, using robust energy methods in the Lagrangian formulation. Following these breakthrough results, there has been considerable amount of work on the local well-posedness theory of one-fluid and two-fluid interfaces, and other related systems.
We refer the reader to \cite{ABZ1,Am,AmMa,ChCoSh,ChHuSt,ChLi,CoSh,IgTaTa,La1,La,Li,Pu,ShZe1,ShZe2,ShZe3} for some of the works on the local well-posedness theory of the system, at  various levels of generality,\footnote{We remark that the surface tension coefficient $\sigma$ has to be strictly positive in order to prevent Kelvin--Helmholtz instabilities and have local well-posedness, in the case of two-fluid interfaces $\rho^1>0$, $\rho^2>0$.} and to the book \cite{LaBook} for a comprehensive treatment of the water waves problem.

\subsubsection{Formation of singularities} There are at least two types of singularities one could imagine: (1) loss of local regularity of the interface, or (2) self-intersection of the interface. It appears that very little is known about the possibility of dynamical loss of regularity of solutions, or even about the possibility of significant norm growth. On the other hand, Castro--C\'{o}rdoba--Fefferman--Gancedo--G\'{o}mez-Serrano \cite{CaCoFeGaGo} were able to construct "splash" singularities, i.e. locally smooth interfaces that self-intersect in finite time, in the 2D water waves problem with no surface tension and irrotational flows ($\rho^2=0$, $\sigma=0$). This was later extended to the case with surface tension $\sigma\geq 0$, see \cite{CaCoFeGaGo2}, and to the case of two-dimensional one-fluid interfaces and nontrivial vorticity \cite{CoSh2}.

The main goal of this paper is to prove that such "splash" singularities cannot develop dynamically in the case of a two-fluid evolution, with $\rho^1>0,\rho^2>0$, at least if one considers only irrotational flows.

\subsubsection{Global regularity} The question of global regularity in the case of two-fluid evolutions \eqref{cv1}--\eqref{cv4} is very much open. In fact, in most cases, it is not known if there are {\it{any}} nontrivial smooth global solutions. The first nontrivial smooth global solutions were constructed independently by Wu \cite{Wu3DWW} and Germain--Masmoudi--Shatah \cite{GMS2} in the case of the gravity water waves problem in 3D (two dimensional interfaces). With our notation this corresponds to the case $\sigma=0$ and $\rho^2=0$. More recently, Germain--Masmoudi--Shatah also constructed global solutions for the capillary water waves equation in dimension 3, which corresponds to $g=0$ and $\rho^2=0$.

In the more difficult\footnote{The mechanism used to prove global existence relies in a critical way on dispersion. These dispersion effects are stronger in the case of two-dimensional interfaces than in one-dimensional interfaces, which leads to faster pointwise decay in time of solutions.} case of one-dimensional interfaces, Wu \cite{WuAG} investigated the long time behavior of small irrotational gravity waves ($\sigma=0$, $\rho^2=0$), and proved almost global existence of solutions. Very recently, the first nontrivial global solutions for the gravity water waves problem in 2D were constructed by Ionescu--Pusateri \cite{IoPu1}, \cite{IoPu} and, slightly later and independently, by Alazard--Delort \cite{AlDe1}, \cite{AlDe2}.

We remark that all the global solutions constructed so far are small and irrotational perturbations of the rest solution, in the water one-fluid models corresponding to  $\rho^2=0$. It would be very interesting to construct (1) smooth global solutions with nontrivial vorticity, or (2) smooth global solutions in the more stable case of two-fluid interfaces $\rho^1>0$, $\rho^2>0$.

\subsection{The main theorem}\label{MainStatement} Our problem can be analyzed in two settings, depending on the conditions imposed on the boundary parametrization $z$:

\begin{itemize}
\item {Asymptotically flat: in this case $z(\alpha,t)-(\alpha,0)\,\rightarrow\,0$ in a suitable sense as $\alpha\,\rightarrow\,\infty$.}

\item {Periodic: in this case $z(\alpha,t)-(\alpha,0)$ is a $2 \pi$ periodic function of $\alpha$.}

\end{itemize}

In what follows we will only consider the periodic case. In this case, letting $\mathrm{Tr}_2:\mathbb{R}^2\to\mathbb{R}^2$, $\mathrm{Tr}_2(z_1,z_2):=(z_1+2\pi,z_2)$, we assume that, for any $t\in I$,
\begin{equation}\label{syst2}
\begin{split}
&z(\alpha+2\pi,t)=\mathrm{Tr}_2(z(\alpha,t)),\qquad \mathrm{Tr}_2(\Omega^1(t))=\Omega^1(t),\qquad \mathrm{Tr}_2(\Omega^2(t))=\Omega^2(t),\\
&u^j(\mathrm{Tr}_2(z),t)=u^j(z,t),\qquad j\in\{1,2\},\,z\in \Omega^j(t),\\
&\|u^1(t)\|_{L^2(\Omega^1(t)\cap [0,2\pi]\times\mathbb{R})}+\|u^2(t)\|_{L^2(\Omega^2(t)\cap [0,2\pi]\times\mathbb{R})}<\infty.
\end{split}
\end{equation}

We are now ready to state our main theorem:

\begin{theorem}\label{MainThm}
Assume $(z,u^1,u^2)$ is a solution of the two-fluid system 
\begin{equation}\label{syst} 
\begin{cases}
&\rho^{j} (\partial_t u^{j}+[u^{j}\cdot \nabla]u^{j})=-\nabla P^{j}-g\rho^{j}e_2,\qquad j\in\{1,2\},\\
&(\partial_tz-u^j|_{\Gamma_t})\cdot(\partial_\alpha z)^\perp=0,\qquad j\in\{1,2\},\\
&(P^2-P^1)|_{\Gamma_t}=-\sigma K(z)=-\sigma\frac{\partial_\alpha z_1\partial_\alpha^2z_2-\partial_\alpha z_2\partial_\alpha^2z_1}{[(\partial_\alpha z_1)^2+(\partial_\alpha z_2)^2]^{3/2}},\\
&\mathrm{div}\,u^j=0,\qquad\mathrm{curl}\,u^j=0,\qquad j\in\{1,2\},
\end{cases}
\end{equation}
in the periodic case \eqref{syst2}, on some time interval $[0,T]$. Assume that $\rho^1,\rho^2>0$ and assume that there is a constant $A>0$ such that

\begin{itemize}

\item {$z:\:\mathbb{R}\times[0,T]\,\rightarrow\,\mathbb{R}^2$ is a regular parametrization, i.e.
\begin{equation}\label{zsmooth}
\|z(\alpha,t)-(\alpha,0)\|_{C^{4} (\mathbb{R}\times[0,T])}\leq A\qquad\text{ and }\inf_{\alpha\in\mathbb{R},\,t\in[0,T]}|z_{\alpha}(\alpha,t)|\geq A^{-1}.
\end{equation}}

\item {The velocity field $u^1$ is smooth on the interface, i.e.
\begin{equation}\label{u1smooth}
\|v^1\|_{C^{3} (\mathbb{R}\times[0,T])}\leq A,
\end{equation}
where $v^1(\alpha,t):=u^1(z(\alpha,t),t)$.}

\item {Letting
\begin{equation}\label{defchart}
F(z)(\al,\be,t):=
\begin{cases}
\frac{|z(\al,t)-z(\be,t)|}{|\al-\be|}\qquad &\text{ if }\alpha\neq\beta,\\
|z_\alpha(\alpha,t)|\qquad &\text{ if }\alpha=\beta,
\end{cases}
\end{equation}
we assume that, at time $t=0$, the parametrization $z$ obeys the chord-arc condition
\begin{equation}\label{charc0}
\inf_{\alpha,\beta\in\R}|F(z)(\al,\be,0)|\geq 1/A.
\end{equation}
}

\item {At time $t=0$, the velocity of the second fluid at the boundary $\Gamma_0$ is an $L^{\infty}$ function, i.e.
\begin{equation}\label{u20bound}
\sup_{\alpha\in\mathbb{R}}|u^2(z(\al,0),0)|\leq A.
\end{equation}}
\end{itemize}

Then the {\bf{qualitative}} assumption
\begin{equation}\label{qual}
\inf_{\alpha,\beta\in\R,\,t\in[0,T]}|F(z)(\al,\be,t)|>0
\end{equation}
implies the {\bf{quantitative}} bound
\begin{equation}\label{quant}
\inf_{\alpha,\beta\in\R,\,t\in[0,T]}|F(z)(\al,\be,t)|\geq c(A,T),
\end{equation}
for some constant $c(A,T)>0$ that depends only on $A$, $T$, and the constants $\sigma,g,\rho^1,\rho^2$.
\end{theorem}

In other words, "splash" singularities cannot develop smoothly in the case of (locally) regular solutions of the two-fluid interface system. Indeed, if a splash were to develop at time $t_{\mathrm{splash}}$ then one can apply the theorem on the time interval $[0,t_{\mathrm{splash}}-\eps]$, for any $\eps>0$, to conclude that the chord-arc function $F(z)$ is bounded from below, uniformly in $\eps$. By continuity, $F(z)$ is bounded from below at time $t=t_{\mathrm{splash}}$, contradiction.

This is in sharp contrast with the situation in the water waves system studied in \cite{CaCoFeGaGo}, which corresponds to $\rho^2=0$. On the other hand, a similar result ruling out the possibility of splash singularities was proved recently by Gancedo--Strain, see \cite{GaSt} and the references therein, in the case of certain solutions of the SQG equation and the Muskat problem. 

\begin{remark}\label{remarks} (i) The periodicity assumptions are only used to derive the integral formulas \eqref{ko1}, which are later used to prove bounds such as \eqref{ko102} and localize to a small neighborhood of the potential splash. Suitable asymptotic flatness assumptions would also be sufficient for these purposes. The rest of the argument relies only on the bounds \eqref{zsmooth}--\eqref{u1smooth}.

(ii) The condition $\rho^1,\rho^2>0$ is only used in Lemma \ref{vort}, to prove the critical $L^\infty$ bound on the boundary vorticity $\omega$. We note that the mechanism of splash formation in \cite{CaCoFeGaGo} in the water-wave problem (which corresponds to $\rho^2=0$) requires blowup of the boundary vorticity.

(iii) Our theorem does not prevent the dynamical formation of self-intersection singularities in all cases, and it might be possible that the interface will self-intersect in finite time in certain cases. Our theorem only states that this can only happen if the interface losses smoothness as one approaches the time of self-intersection, i.e. the two collapsing waves will become sharper as they approach each other. Intuitively, this sharpening of the waves would allow the fluid in between to escape easier, and the collapsing waves would still be able to travel towards each other and intersect in finite time. 

It seems quite difficult, however, to determine rigorously whether such a scenario can occur. One would have to understand, at the very least, how to dynamically create interfaces with very large high Sobolev norms.

(iv) The proof shows that the dependence of the constant $c(A,T)$ in \eqref{quant} on $T$ is at most double exponential, i.e. $c(A,T)\geq e^{-Ce^{CT}}$ for some constant $C$ that depends only on $A,\sigma,g,\rho^1,\rho^2$.  
\end{remark}

\subsubsection{Notation} Given two nonnegative quantities $X,Y$, the relation
\begin{equation*}
X\lesssim Y
\end{equation*}
means that there is a number $C$ that may depend only on the constants $A,\sigma,g,\rho^1,\rho^2$ in Theorem \ref{MainThm} such that $X\leq CY$. Similarly, the notation \begin{equation*}
X\approx Y
\end{equation*}
means $X\lesssim Y$ and $Y\lesssim X$.

The rest of the paper is concerned with the proof of Theorem \ref{MainThm} and is organized as follows. In section \ref{MainThm1} we start our analysis: we first rewrite our equations in terms of the boundary vorticity $\omega$, using the fact that we are in the irrotational case, see Proposition \ref{interface}. Then we prove that this boundary vorticity is bounded, see Lemma \ref{vort}. We then show that a potential splash could only develop in a certain controlled way: two different points on the interface can only approach dynamically if the segment between them is contained in the domain $\Omega^2$, and the two arcs on the interface become almost parallel. See Lemma \ref{chord-arc2} and Proposition \ref{chord-arc3}.

Our goal is to show that the difference of the velocities of the points that could potentially cause a splash becomes very small, as these points approach. For this we use the formula \eqref{ko1}, which expresses these velocities as singular integral operators in terms of $\omega$, and the apriori quantitative $L^\infty$ bound on $\omega$. The argument, which is the most technical part of the paper is carried out in section \ref{MainProp}, using a suitably defined auxiliary $Z$ norm. See Proposition \ref{TechProp}.

In section \ref{MainThm2} we complete the proof of the main theorem, essentially using Gronwall's inequality.

\section{Proof of the main theorem, I: preliminary reductions}\label{MainThm1}

\subsection{The equations on the interface}\label{reduct1}

Given that the velocity fields $u^1(t)$ and $u^2(t)$ are assumed to be both incompressible and irrotational in their respective domains, it is well-known that the system \eqref{syst}--\eqref{syst2} can be reduced to a system on the interface. For the sake of completeness we derive this reduction in this subsection. More precisely, we prove the following:

\begin{proposition}\label{interface}
Let $v^j(\alpha,t):=u^j(z(\alpha,t),t)$ and $\widetilde{P}^j(\alpha,t):=P^j(z(\alpha,t),t)$, $j\in\{1,2\}$. Then, with the assumptions in Theorem \ref{MainThm}, there is a $2\pi$-periodic $C^2$ function $\omega$ (the vorticity) such that
\begin{equation}\label{ko1}
\begin{split}
&v^1=BR(z,\omega)+\frac{\omega}{2}\frac{z_\alpha}{|z_\alpha|^2},\qquad v^2=BR(z,\omega)-\frac{\omega}{2}\frac{z_\alpha}{|z_\alpha|^2},\\
&BR(z,\omega)(\alpha,t):=\frac{1}{2\pi}p.v.\int_{\mathbb{R}}\frac{(z(\alpha,t)-z(\beta,t))^\perp}{|z(\alpha,t)-z(\beta,t)|^2}\omega(\beta,t)\,d\beta.
\end{split}
\end{equation}
Moreover,
\begin{equation}\label{ko2}
\begin{split}
&[z_t-BR(z,\omega)]\cdot (z_\alpha)^\perp=0,\\
&\partial_t(v^j\cdot z_\alpha)-\partial_\alpha(v^j\cdot z_t)+\frac{1}{2}\partial_\alpha(|v^j|^2)+\frac{1}{\rho^j}\partial_\alpha\widetilde{P}^j+g\partial_\alpha z_2=0,\qquad j\in\{1,2\}.
\end{split}
\end{equation}
\end{proposition}

\begin{proof}[Proof of Proposition \ref{interface}] In view of the second equation in \eqref{syst}, we have $(v^1-v^2)\cdot z_\alpha^\perp=0$. Let
\begin{equation}\label{ko3}
\omega:=(v^1-v^2)\cdot z_\alpha,
\end{equation}
and notice that $\omega\in C^2(\mathbb{R}\times I:\mathbb{R})$ is a $2\pi$ periodic function.

We prove first the identities in the first line of \eqref{ko1}. For $j\in\{1,2\}$ and $x\in \Omega^j(t)$, we define
\begin{equation}\label{ko3.5}
U^j(x,t):=\frac{1}{2\pi}p.v.\int_{\mathbb{R}}\frac{(x-z(\beta,t))^\perp}{|x-z(\beta,t)|^2}\omega(\beta,t)\,d\beta.
\end{equation}
For simplicity of notation, we assume that $t\in[0,T]$ is fixed in the rest of this argument, and drop the $t$ dependence of the functions. By definition, $U^j\in C^4(\Omega^j:\mathbb{R}^2)$ for any $j\in\{1,2\}$. In addition,
\begin{equation*}
\mathrm{div}\,U^j=0\quad\text{ and }\quad\mathrm{curl}\,U^j=0\quad \text{ in }\quad \Omega^j.
\end{equation*}
Using the assumptions on the domains $\Omega^j$, it is easy to see that there is $\eps_0>0$ such that
\begin{equation*}
\begin{split}
&z(\alpha)-\eps [z_\alpha(\alpha)]^\perp\in\Omega^1\text{ for any }\alpha\in\mathbb{R}\text{ and }\eps\in(0,\eps_0),\\
&z(\alpha)+\eps [z_\alpha(\alpha)]^\perp\in\Omega^2\text{ for any }\alpha\in\mathbb{R}\text{ and }\eps\in(0,\eps_0).
\end{split}
\end{equation*}
Therefore, using standard manipulations of singular integrals,
\begin{equation}\label{ko4}
\begin{split}
&\lim_{\eps\to 0}U^1(z(\alpha)-\eps [z_\alpha(\alpha)]^\perp)=BR(z,\omega)(\alpha)+\frac{\omega(\alpha)}{2}\frac{z_\alpha(\alpha)}{|z_\alpha(\alpha)|^2},\\
&\lim_{\eps\to 0}U^2(z(\alpha)+\eps [z_\alpha(\alpha)]^\perp)=BR(z,\omega)(\alpha)-\frac{\omega(\alpha)}{2}\frac{z_\alpha(\alpha)}{|z_\alpha(\alpha)|^2},
\end{split}
\end{equation}
uniformly for any $\alpha\in\mathbb{R}$, where the Birkhoff--Rott operators $BR$ are defined in \eqref{ko1}. 

In addition, the vector-fields $U^j-u^j$ are smooth in $\Omega^j$ and satisfy the equations $\mathrm{div}\,(U^j-u^j)=0,\,\mathrm{curl}\,(U^j-u^j)=0$ in $\Omega^j$. Let
\begin{equation*}
f(x):=\begin{cases}
(U^1_2-u^1_2)+i(U^1_1-u^1_1)&\qquad\text{ if }x\in\Omega^1;\\
(U^2_2-u^2_2)+i(U^2_1-u^2_1)&\qquad\text{ if }x\in\Omega^2,
\end{cases}
\end{equation*}
and notice that $f$ is a complex analytic function in $\Omega^1\cup\Omega^2$. Moreover, in view of \eqref{ko4},
\begin{equation*}
\begin{split}
\lim_{\eps\to 0}f(z(\alpha)-\eps [z_\alpha(\alpha)]^\perp)&=\lim_{\eps\to 0}f(z(\alpha)+\eps [z_\alpha(\alpha)]^\perp)\\
&=(BR(z,\omega)_2+i BR(z,\omega)_1)+\frac{\omega\cdot\partial_\alpha(z_2+iz_1)}{2|\partial_\alpha z|^2}-(v^1_2+iv^1_1).
\end{split}
\end{equation*}
uniformly for any $\alpha\in\mathbb{R}$. Therefore, $f$ extends to a complex analytic function in $\mathbb{R}^2$. Moreover, $f\circ\mathrm{Tr}_2=f$, $f$ is bounded in $\mathbb{R}^2$, and $\lim_{|x_2|\to \infty}|f(x)|=0$ (using the assumption \eqref{syst2} and the definition \eqref{ko3.5}). Therefore $f\equiv 0$ in $\mathbb{R}^2$, and the desired identities in the first line of \eqref{ko1} follow.

We prove now the identities \eqref{ko2}. The identity in the first line is a simple consequence of \eqref{ko1} and the assumption $(\partial_tz-u^1|_{\Gamma_t})\cdot(\partial_\alpha z)^\perp=0$ in \eqref{syst}. To prove the identities in the second line, we may assume that $j=1$ and calculate, using the definitions,
\begin{equation*}
\begin{split}
\partial_t(v^1\cdot z_\alpha)(\alpha,t)&=\sum_{k=1}^2v^1_k(\alpha,t)(\partial_\alpha\partial_tz_k)(\alpha,t)+\sum_{k=1}^2(\partial_\alpha z_k)(\alpha,t)(\partial_tu_k^1)(z(\alpha,t),t)\\
&+\sum_{k,l=1}^2(\partial_\alpha z_k)(\alpha,t)(\partial_lu_k^1)(z(\alpha,t),t)(\partial_tz_l)(\alpha,t),
\end{split}
\end{equation*}
\begin{equation*}
\partial_\alpha(v^1\cdot z_t)(\alpha,t)=\sum_{k=1}^2v^1_k(\alpha,t)\partial_\alpha\partial_tz_k(\alpha,t)+\sum_{k,l=1}^2(\partial_tz_k)(\alpha,t)(\partial_lu_k^1)(z(\alpha,t),t)\partial_\alpha z_l(\alpha,t),
\end{equation*}
and
\begin{equation*}
\frac{1}{2}\partial_\alpha(|v^1|^2)(\alpha,t)=\sum_{k,l=1}^2u^1_k(z(\alpha,t),t)(\partial_lu^1_k)(z(\alpha,t),t)(\partial_\alpha z_l)(\alpha,t).
\end{equation*}
Therefore, using the first identity in \eqref{syst} and the identity $\partial_lu_k^1=\partial_ku^1_l$,
\begin{equation*}
\begin{split}
&\rho^1\Big[\partial_t(v^1\cdot z_\alpha)(\alpha,t)-\partial_\alpha(v^1\cdot z_t)(\alpha,t)+\frac{1}{2}\partial_\alpha(|v^1|^2)(\alpha,t)\Big]\\
&=\rho^1\Big[\sum_{k=1}^2(\partial_\alpha z_k)(\alpha,t)(\partial_tu_k^1)(z(\alpha,t),t)+\sum_{k,l=1}^2u^1_k(z(\alpha,t),t)(\partial_lu^1_k)(z(\alpha,t),t)(\partial_\alpha z_l)(\alpha,t)\Big]\\
&=-\sum_{k=1}^2(\partial_\alpha z_k)(\alpha,t)(\partial_kP^1)(z(\alpha,t),t)-g\rho^1(\partial_\alpha z_2)(\alpha,t).
\end{split}
\end{equation*}
The desired identity in the second line of \eqref{ko2} follows.
\end{proof}

\subsection{Boundedness of the boundary vorticity}\label{vortbound}

In this subsection we prove a key $L^\infty$ bound on the boundary vorticity $\omega$. This quantitative bound uses the assumption $\rho^2>0$, and is the main difference between the model studied in this paper and the water wave system studied in \cite{CaCoFeGaGo}.

\begin{lemma}\label{vort}
With the notation in Proposition \ref{interface}, we have
\begin{equation*}
\|\omega\|_{L^\infty(\mathbb{R}\times[0,T])}\lesssim 1.
\end{equation*}
\end{lemma}

\begin{proof}[Proof of Lemma \ref{vort}] We use the identities in Proposition \ref{interface}. We have
\begin{equation*}
v^2=v^1-\omega\frac{z_\alpha}{|z_\alpha|^2}.
\end{equation*}
Therefore
\begin{equation*}
\begin{split}
v^2\cdot z_\alpha=v^1\cdot z_\alpha-\omega,\qquad v^2\cdot z_t=v^1\cdot z_t-\omega\frac{z_t\cdot z_\alpha}{|z_\alpha|^2},\qquad |v^2|^2=|v^1|^2+\frac{\omega^2}{|z_\alpha|^2}-2\omega\frac{v^1\cdot z_\alpha}{|z_\alpha|^2}.
\end{split}
\end{equation*}
The equations in the second line of \eqref{ko2} show that
\begin{equation*}
\begin{split}
\partial_t[(\rho^1v^1-\rho^2v^2)\cdot z_\alpha)]&-\partial_\alpha[(\rho^1v^1-\rho^2v^2)\cdot z_t]+\frac{1}{2}\partial_\alpha[\rho^1|v^1|^2-\rho^2|v^2|^2]\\
&+\partial_\alpha(\widetilde{P}^1-\widetilde{P}^2)+g(\rho^1-\rho^2)\partial_\alpha z_2=0.
\end{split}
\end{equation*}
We use the formulas above to conclude that $\omega$ satisfies the modified Burgers' equation
\begin{equation}\label{ko7}
\partial_t\omega=\frac{1}{2}\partial_\alpha\Big[\frac{\omega^2}{|z_\alpha|^2}\Big]+\partial_\alpha(F_1\cdot \omega)+F_2,
\end{equation}
where
\begin{equation*}
\begin{split}
&F_1:=\frac{(z_t-v^1)\cdot z_\alpha}{|z_\alpha|^2},\\
&F_2:=\frac{\partial_\alpha(\widetilde{P}^2-\widetilde{P}^1)}{\rho^2}+\frac{g(\rho^2-\rho^1)\partial_\alpha z_2}{\rho^2}+\frac{\rho^2-\rho^1}{\rho^2}\partial_t(v^1\cdot z_\alpha)+\frac{\rho^2-\rho^1}{2\rho^2}\partial_\alpha(v^1\cdot v^1-2v^1\cdot z_t).
\end{split}
\end{equation*}
Recalling also the assumption, see \eqref{syst},
\begin{equation*}
(\widetilde{P}^2-\widetilde{P}^1)=-\sigma\frac{\partial_\alpha z_1\partial_\alpha^2z_2-\partial_\alpha z_2\partial_\alpha^2z_1}{[(\partial_\alpha z_1)^2+(\partial_\alpha z_2)^2]^{3/2}},
\end{equation*}
and the regularity assumption \eqref{zsmooth}--\eqref{u1smooth}, it follows that $F_1,F_2$ are $2\pi$-periodic functions that satisfy
\begin{equation}\label{ko8}
\|F_1\|_{C^2(\mathbb{R}\times[0,T])}+\|F_2\|_{C^1(\mathbb{R}\times[0,T])}\lesssim 1.
\end{equation}

For any $t\in[0,T]$ and any even integer $p\geq 2$ let
\begin{equation*}
I_p(t):=\int_{\mathbb{T}}\frac{\omega^p}{(z_\alpha\cdot z_\alpha)^{(p-1)/2}}\,d\alpha.
\end{equation*}
Using \eqref{ko7}, \eqref{ko8}, and \eqref{zsmooth}, we estimate
\begin{equation*}
\begin{split}
\Big|\frac{d}{dt}I_p(t)\Big|&\leq\Big|\int_{\mathbb{T}}\frac{p\omega^{p-1}}{(z_\alpha\cdot z_\alpha)^{(p-1)/2}}\Big[\frac{1}{2}\partial_\alpha\Big[\frac{\omega^2}{z_\alpha\cdot z_\alpha}\Big]+\partial_\alpha(F_1\cdot \omega)+F_2\Big]\,d\alpha\Big|\\
&+\Big|\int_{\mathbb{T}}\frac{(p-1)\omega^{p}\partial_t[(z_\alpha\cdot z_\alpha)^{1/2}]}{(z_\alpha\cdot z_\alpha)^{p/2}}\,d\alpha\Big|\\
&\leq Cp[I_p(t)+1],
\end{split}
\end{equation*}
where $C\geq 1$ is a constant that may depend only on $A$. Moreover, $|I_p(0)|\leq C^p$ (see the formula \eqref{ko3} and the assumption \eqref{u20bound}). Therefore, using Gronwall's inequality, $I_p(t)\leq e^{Cp(T+1)}$ for any $t\in[0,T]$ and any even integer $p\geq 2$, and the desired conclusion follows by letting $p\to\infty$. 
\end{proof}

\subsection{Restricted chord-arc bounds}\label{Fluid1}

In our next lemma we prove some simple properties of the chord-arc function $F(z)$ defined in \eqref{defchart}. In particular we study the case when $F(z)$ becomes small at some time $t$.

\begin{lemma}\label{chord-arc}
(i) There is $\varep_1=\varep_1(A)>0$ sufficiently small such that
\begin{equation}\label{ko15}
F(z)(\alpha,\beta,t)\geq \varep_1\qquad\text{ if }\qquad |\alpha-\beta|\leq \varep_1,\,(\alpha,\beta,t)\in\mathbb{R}\times\mathbb{R}\times[0,T].
\end{equation}

(ii) There is a constant $\varep_2=\varep_2(A)$ sufficiently small with the following property: if 
\begin{equation}\label{ko16}
\inf_{\alpha,\beta\in\mathbb{R}}F(z)(\alpha,\beta,t)=\varep\leq \varep_2,
\end{equation}
for some $t\in[0,T]$ then there are points $\alpha_1,\alpha_2\in\mathbb{R}$ such that
\begin{equation}\label{ko17}
\begin{split}
&|\alpha_1-\alpha_2|\in [2\varep_1,\varep_1^{-1}/2];\\
&|z(\alpha_1,t)-z(\alpha_2,t)|=\inf_{\alpha,\beta\in\mathbb{R},\,|\alpha-\beta|\in[\varep_1,\varep_1^{-1}]}|z(\alpha,t)-z(\beta,t)|\approx\varep.
\end{split}
\end{equation}
Moreover
\begin{equation}\label{ko17.1}
(z(\alpha_1,t);z(\alpha_2,t))\cap\Gamma_t=\emptyset,
\end{equation}
\begin{equation}\label{ko17.2}
[z(\alpha_1,t)-z(\alpha_2,t)]\cdot z_\alpha(\alpha_1,t)=[z(\alpha_1,t)-z(\alpha_2,t)]\cdot z_\alpha(\alpha_2,t)=0,
\end{equation}
and
\begin{equation}\label{ko17.3}
z_\alpha(\alpha_1,t)\cdot z_\alpha(\alpha_2,t)<0,
\end{equation}
where $(p;q)$ denotes the open line segment in the plane between the points $p$ and $q$. 

A pair of points $z(\alpha_1,t),z(\alpha_2,t)$ satisfying \eqref{ko17} will be called a conjugate pair.
\end{lemma}

\begin{proof}[Proof of Lemma \ref{chord-arc}] The conclusion \eqref{ko15} is a simple consequence of the assumptions in \eqref{zsmooth} and Taylor's formula.

In proving part (ii), we assume that $t$ is fixed and drop it from the notation. We fix $\alpha_1,\alpha_2\in\mathbb{R}$ such that
\begin{equation}\label{ko26}
|z(\alpha_1)-z(\alpha_2)|=\inf_{\alpha,\beta\in\mathbb{R},\,|\alpha-\beta|\in[3\varep_1,\varep_1^{-1}/3]}|z(\alpha)-z(\beta)|.
\end{equation}
In view of the assumption \eqref{ko16} and part (i), there are $\alpha_0,\beta_0\in\mathbb{R}$ with $|\alpha_0-\beta_0|\geq\eps_1$ such that $F(z)(\alpha_0,\beta_0)\leq 2\varep$. It follows that 
\begin{equation*}
\inf_{\alpha,\beta\in\mathbb{R},\,|\alpha-\beta|\in[\varep_1,\varep_1^{-1}]}|z(\alpha)-z(\beta)|\lesssim \varep.
\end{equation*}
On the other hand, if $|\alpha-\beta|\in[\varep_1,3\varep_1]\cup [\varep_1^{-1}/3,\varep_1^{-1}]$ then, using the assumption \eqref{zsmooth} and Taylor's formula, $|z(\alpha)-z(\beta)|\gtrsim \varep_1$. Therefore
\begin{equation}\label{ko27}
\inf_{\alpha,\beta\in\mathbb{R},\,|\alpha-\beta|\in[\varep_1,\varep_1^{-1}]}|z(\alpha)-z(\beta)|=\inf_{\alpha,\beta\in\mathbb{R},\,|\alpha-\beta|\in[3\varep_1,\varep_1^{-1}/3]}|z(\alpha)-z(\beta)|,
\end{equation}
and \eqref{ko17} follows.

The identity \eqref{ko17.1} follows from the definition and \eqref{ko27}. The identities \eqref{ko17.2} also follow from \eqref{ko27}, since the function $\alpha\to |z(\alpha)-z(\alpha_2)|^2$ is locally minimized when $\alpha=\alpha_1$, and the function $\alpha\to |z(\alpha_1)-z(\alpha)|^2$ is locally minimized when $\alpha=\alpha_2$.

We prove now the inequality \eqref{ko17.3}. We may assume that $\alpha_1<\alpha_2$ and consider the curve $\gamma:[0,2]\to\mathbb{R}^2$,
\begin{equation*}
\gamma(\mu):=
\begin{cases}
z((1-\mu)\alpha_1+\mu\alpha_2),\qquad&\text{ if }\mu\in[0,1];\\
(2-\mu)z(\alpha_2)+(\mu-1)z(\alpha_1),\qquad&\text{ if }\mu\in[1,2].
\end{cases}
\end{equation*}
Therefore, $\gamma$ travels along the curve $z$ from $z(\alpha_1)$ to $z(\alpha_2)$, and then back along the line segment $[z(\alpha_2);z(\alpha_1)]$. Clearly, $\gamma$ is a simple closed curve in the plane. Therefore, using the Jordan curve theorem, $\mathbb{R}^2\setminus\gamma=O_1\cup O_2$ where $O_1,O_2$ are disjoint connected open sets, $O_2$ is bounded, $O_1$ is unbounded, and $\gamma=\partial O_1=\partial O_2$. The curves $\{z(\alpha):\alpha\in(-\infty,\alpha_1)\}$ and $\{z(\alpha):\alpha\in(\alpha_2,\infty)\}$ are unbounded curves in the plane (using the periodicity assumption), and do not intersect the curve $\gamma$ (using \eqref{ko17.1}). Therefore
\begin{equation}\label{ko28}
z(\alpha)\in O_1\qquad\text{ for any }\alpha\in\mathbb{R}\setminus[\alpha_1,\alpha_2].
\end{equation}

Let $v:=z(\alpha_2)-z(\alpha_1)$. For $\delta>0$ we consider the sets
\begin{equation*}
\begin{split}
R_\delta^-:=\{z(\alpha)+\mu v:\,\alpha\in(\alpha_1-\delta,\alpha_1),\,\mu\in[0,2/3]\},\\
R_\delta^+:=\{z(\alpha)-\mu v:\,\alpha\in(\alpha_2,\alpha_2+\delta),\,\mu\in[0,2/3]\}.
\end{split}
\end{equation*}
The sets $R_\delta^{\pm}$ are clearly connected. Moreover, if $\delta$ is sufficiently small then
\begin{equation}\label{ko28.1}
R_\delta^-\cap\gamma=R_\delta^+\cap\gamma=\emptyset.
\end{equation}
Indeed, to prove \eqref{ko28.1} assume for contradiction that there is a point $p\in R_\delta^-\cap\gamma$. We have two cases: if $p=z(\beta)$ for some $\beta\in[\alpha_1,\alpha_2]$ then $z(\beta)=z(\alpha)+\mu v$ for some $\alpha\in  (\alpha_1-\delta,\alpha_1)$ and $\mu\in[0,2/3]$. Using the identity in the second line of \eqref{ko17} it follows that $|\alpha-\beta|\leq \varepsilon_1$. Therefore, using also \eqref{ko17.2}, $[z(\beta)-z(\alpha)]\cdot z_\alpha(\alpha_1)=0$ and $\alpha,\beta\in (\alpha_1-\delta,\alpha_1+\varep_1)$. This shows easily that $\alpha=\beta$, which gives a contradiction.

On the other hand, if $p\in (z(\alpha_1);z(\alpha_2))$ then $z(\alpha_1)+\mu_1v=z(\alpha)+\mu v$ for some $\alpha\in  (\alpha_1-\delta,\alpha_1)$, $\mu\in[0,2/3]$, and $\mu_1\in[0,1]$. Using \eqref{ko17.2}, it follows that $[z(\alpha_1)-z(\alpha)]\cdot z_\alpha(\alpha_1)=0$. As before, it follows that $\alpha=\alpha_1$, which is a contradiction. This completes the proof of \eqref{ko28.1}.

Using also \eqref{ko28} it follows that
\begin{equation}\label{ko29}
R_\delta^-\subseteq O_1,\qquad R_\delta^+\subseteq O_1.
\end{equation}
In view of \eqref{ko17.2}, the vectors $z_\alpha(\alpha_1)$ and $z_\alpha(\alpha_2)$ are either parallel or antiparallel. Assume for contradiction that the two vectors are parallel. Then the sets $R_\delta^-$ and $R_\delta^+$ are on opposite sides of the line between the points $z(\alpha_1)$ and $z(\alpha_2)$. Therefore there is a small neighborhood of the point $q=[z(\alpha_1)+z(\alpha_2)]/2$ which is entirely contained in $[z(\alpha_1);z(\alpha_2)]\cup R_\delta^-\cup R_\delta^+$. This is in contradiction with \eqref{ko29} and the assumption $q\in \partial(O_2)$ (which is a consequence of the Jordan curve theorem). It follows that the vectors $z_\alpha(\alpha_1)$ and $z_\alpha(\alpha_2)$ are antiparallel, which gives \eqref{ko17.3}.
\end{proof}

We prove now that the splash scenario can only be reached dynamically through the fluid $2$. More precisely:

\begin{lemma}\label{chord-arc2}
For any $\delta>0$ there is $\varep_3=\varep_3(A,T,\delta)$ sufficiently small with the following property: if $t\in [0,T]$, $\alpha,\beta\in\mathbb{R}$ satisfy  $|\alpha-\beta|\geq\delta$, and
\begin{equation*}
[z(\alpha,t);z(\beta,t)]\subseteq \overline{\Omega^1(t)}
\end{equation*}
then
\begin{equation*}
|z(\alpha,t)-z(\beta,t)|\geq \varep_3.
\end{equation*}
\end{lemma}

\begin{proof}[Proof of Lemma \ref{chord-arc2}] We may assume that $\delta$ is sufficiently small (depending on $A$). Assume, for contradiction, that the conclusion of the lemma fails. Therefore there are points $t_0\in[0,T]$ and $\alpha_0,\beta_0\in\mathbb{R}$ such that
\begin{equation}\label{ko30}
\begin{split}
&|\alpha_0-\beta_0|\geq\delta,\qquad |z(\alpha_0,t_0)-z(\beta_0,t_0)|\leq\varep_3,\\
&[z(\alpha_0,t_0);z(\beta_0,t_0)]\subseteq \overline{\Omega^1(t_0)}.
\end{split}
\end{equation}

We use the assumption \eqref{u1smooth}. The identities $\Delta u^1_m=0$, $m=1,2$, and the maximum principle show that
\begin{equation}\label{ko31}
\sup_{t\in[0,T]}\|u^1(t)\|_{L^\infty(\Omega^1(t))}\lesssim 1.
\end{equation}
Moreover, using $v^1(\alpha,t)=u^1(z(\alpha,t),t)$ and differentiating in $\alpha$,
\begin{equation*}
\partial_\al v^1_m(\alpha,t)=\sum_{j=1}^2(\partial_\al z_j)(\alpha,t)(\partial_ju^1_m)(z(\alpha,t),t),
\end{equation*}
for $m=1,2$. Using also the identities $\partial_1u_1^1+\partial_2u_2^1=0$, $\partial_1u_2^1-\partial_2u_1^1=0$ and the smoothness of the functions $v^1$ and $z$, it follows easily that $\|\partial_ju^1_m(t)\|_{L^\infty(\Gamma_t)}\lesssim 1$. Using the maximum principle,
\begin{equation}\label{ko32}
\sup_{t\in[0,T],\,j,m\in\{1,2\}}\|\partial_ju^1_m(t)\|_{L^\infty(\Omega^1(t))}\lesssim 1.
\end{equation}

Recall the equation for the motion of the interface in \eqref{syst}, which can be written in the form
\begin{equation}\label{ko51}
z_t(\alpha,t)=u^1(z(\alpha,t),t)+A^1(\alpha,t)z_\alpha(\alpha,t),
\end{equation}
for $(\alpha,t)\in\mathbb{R}\times[0,T]$, where $A^1$ is a continuous function satisfying $\|A^1\|_{C^3(\mathbb{R}\times[0,T])}\lesssim 1$. We define Lagrangian label $P=P^1:\overline{\Omega^1(t_0)}\times [0,T]\to\mathbb{R}^2$ as the solution of the transport equation,
\begin{equation}\label{ko50}
P(q,t_0)=q,\qquad \partial_tP(q,t)=u^1(P(q,t),t).
\end{equation}
Using these two equations is it easy to see that $P$ is a well-defined function (in the sense that $P(q,t)\in \overline{\Omega^1(t)}$ for any $q\in \overline{\Omega^1(t_0)}$). Moreover $P(\Omega^1(t_0),t)\subseteq\Omega^1(t)$ and $P(\Gamma_{t_0},t)\subseteq\Gamma_t$; more precisely
\begin{equation}\label{ko52}
P(z(\mu_0,t_0),t)=z(\kappa(\mu_0,t),t),
\end{equation}
for any $\mu_0\in\mathbb{R}$, where $\kappa$ satisfies the transport equation
\begin{equation}\label{ko53}
(\partial_t\kappa)(\mu,t)=-A^1(\kappa(\mu,t),t),\qquad \kappa(\mu,t_0)=\mu.
\end{equation}

In view of \eqref{ko50} and \eqref{ko32},
\begin{equation*}
|P(q,t)-P(q',t)|\lesssim |q-q'|
\end{equation*}
for any $t\in[0,T]$ and any points $q,q'\in\overline{\Omega^1(t_0)}$ with the property that $[q,q']\subseteq\overline{\Omega^1(t_0)}$. In particular, with $\alpha_0,\beta_0$ as in \eqref{ko30},
\begin{equation*}
|P(z(\alpha_0,t_0),0)-P(z(\beta_0,t_0),0)|\lesssim \varep_3.
\end{equation*}
Therefore, using \eqref{ko52},
\begin{equation*}
|z(\kappa(\alpha_0,0),0)-z(\kappa(\beta_0,0),0)|\lesssim\varep_3.
\end{equation*}
On the other hand, using \eqref{ko53}, $|\partial_t\kappa(\alpha_0,t)-\partial_t\kappa(\beta_0,t)|\lesssim |\kappa(\alpha_0,t)-\kappa(\beta_0,t)|$ for any $t\in[0,t_0]$. Recalling also that $|\kappa(\alpha_0,t_0)-\kappa(\beta_0,t_0)|=|\alpha_0-\beta_0|\geq\delta$, see \eqref{ko30}, it follows that
\begin{equation*}
|\kappa(\alpha_0,0)-\kappa(\beta_0,0)|\gtrsim\delta.
\end{equation*}
The last two inequalities are in contradiction with the chord-arc condition \eqref{charc0} at time $t=0$. This completes the proof of the lemma.
\end{proof}

We can now provide a complete description of a neighborhood of a conjugate pair: after a suitable rigid motion, a small neighborhood of a conjugate pair (of size depending quantitatively on $A$ and $T$) can be identified with a neighborhood of the origin in $\mathbb{R}^2$ in such a way that the curve $z(\alpha,t)$, when restricted to this small neighborhood, becomes the union of two smooth graphs that are very close and almost parallel at the origin. Moreover, the domain $\Omega^2(t)$ corresponds to the region between these two curves while the domain $\Omega^1(t)$ corresponds to the complement of this region. More precisely:

\begin{proposition}\label{chord-arc3} (i) There are small constants $\varep_4,\varep_5\in(0,\varep_2)$ with $\varep_4\leq\varep_5^{100}$ (that may depend only on $A,T$) with the following property: assume that $\inf_{\alpha,\beta\in\mathbb{R}}F(z)(\alpha,\beta,t)=\varep\leq\varep_4$ for some $t\in[0,T]$, let $\alpha_1,\alpha_2$ be as in \eqref{ko17}, and let
\begin{equation*}
d:=|z(\alpha_2,t)-z(\alpha_1,t)|,\qquad e:=\frac{z(\alpha_2,t)-z(\alpha_1,t)}{|z(\alpha_2,t)-z(\alpha_1,t)|}.
\end{equation*} 
Then
\begin{equation}\label{ko70}
(z(\alpha_1,t);z(\alpha_2,t))\subseteq \Omega^2(t).
\end{equation}
and
\begin{equation}\label{ko70.5}
z_\alpha(\alpha_1,t)\cdot e=z_\alpha(\alpha_2,t)\cdot e=0,\qquad z_\alpha(\alpha_1,t)\cdot z_\alpha(\alpha_2,t)<0.
\end{equation}
Moreover $d\approx\varep$ and
\begin{equation}\label{ko71}
\begin{split}
&\big\{z(\beta,t)-\rho e:|\beta-\alpha_1|\leq \varep_5,\,\rho\in(0,\varep_5)\big\}\subseteq\Omega^1(t),\\
&\big\{z(\beta,t)+\rho e:|\beta-\alpha_2|\leq \varep_5,\,\rho\in(0,\varep_5)\big\}\subseteq\Omega^1(t),\\
&\big\{z(\beta,t)+\rho e:|\beta-\alpha_1|\leq \varep_5,\,\rho\in(0,d)\big\}\subseteq\Omega^2(t),\\
&\big\{z(\beta,t)-\rho e:|\beta-\alpha_2|\leq \varep_5,\,\rho\in(0,d)\big\}\subseteq\Omega^2(t).
\end{split}
\end{equation}

(ii) Let $R:\mathbb{R}^2\to\mathbb{R}^2$ denote the rigid motion transformation that satisfies 
\begin{equation}\label{ko74}
R(0,0)=\frac{z(\alpha_1,t)+z(\alpha_2,t)}{2},\quad R(0,-d/2)=z(\alpha_1,t),\quad R(0,d/2)=z(\alpha_2,t).
\end{equation}
Then there are functions $f_1,f_2,\beta_1,\beta_2:(-\varep_5^4,\varep_5^4)\to (-\varep_5^2,\varep_5^2)$ such that
\begin{equation}\label{ko75}
R(\rho,f_1(\rho))=z(\alpha_1+\beta_1(\rho),t),\qquad R(\rho,f_2(\rho))=z(\alpha_2+\beta_2(\rho),t)
\end{equation}
for any $\rho\in(-\varep_5^4,\varep_5^4)$. Moreover
\begin{equation}\label{ko76}
\begin{split}
&f_1(0)=-d/2,\quad f_2(0)=d/2,\quad\beta_1(0)=\beta_2(0)=0,\quad f'_1(0)=f'_2(0)=0,\\
&\|f_1\|_{C^4}+\|f_2\|_{C^4}+\|\beta_1\|_{C^4}+\|\beta_2\|_{C^4}\lesssim 1,\\
&-\beta'_1(\rho)\beta'_2(\rho)\approx 1\text{ and } f_2(\rho)-f_1(\rho)\geq d\qquad\text{ for all }\rho\in(-\varep_5^4,\varep_5^4).
\end{split}
\end{equation}
Finally,
\begin{equation}\label{ko72}
\begin{split}
&R\Big(\big\{(\rho,y):\rho\in(-\varep_5^4,\varep_5^4),\,y\in (f_2(\rho),\varep_5^2)\big\}\Big)\subseteq\Omega^1(t),\\
&R\Big(\big\{(\rho,y):\rho\in(-\varep_5^4,\varep_5^4),\,y\in (-\varep_5^2,f_1(\rho))\big\}\Big)\subseteq\Omega^1(t),\\
&R\Big(\big\{(\rho,y):\rho\in(-\varep_5^4,\varep_5^4),\,y\in (f_1(\rho),f_2(\rho))\big\}\Big)\subseteq\Omega^2(t).
\end{split}
\end{equation}
\end{proposition}

\begin{proof}[Proof of Proposition \ref{chord-arc3}] As before, we assume that $t$ is fixed and drop it from the notation.

The conclusions in part (i) follow easily from Lemma \ref{chord-arc} and Lemma \ref{chord-arc2}. Indeed, the existence of the points $\alpha_1,\alpha_2$ satisfying \eqref{ko17} and $|z(\alpha_2)-z(\alpha_1)|\approx \varep$ follows from Lemma \ref{chord-arc} (ii), provided that $\varep_4\leq\varep_2$. The conclusion \eqref{ko70} then follows from \eqref{ko17.1} and Lemma \ref{chord-arc2}, provided that $\delta=\varep_1$ and $\varep_4$ is sufficiently small. 

The conclusions in \eqref{ko70.5} follow from \eqref{ko17.1} and \eqref{ko17.2}.

To prove \eqref{ko71} we notice first that $\{z(\beta)+\rho' e:|\beta-\alpha_1|\leq 2\varep_5,\,\rho'\in(0,\delta')\}\subseteq \Omega^2$, provided that $\delta'$ is sufficienly small. This is due to the fact that $z$ is a regular curve in a neighborhood of $\alpha_1$, which does not self-intersect, the assumption $\Omega^2$ connected, and \eqref{ko70}. Since $\Gamma_t\subseteq\partial\Omega^1$, it follows that $\{z(\beta)-\rho' e:|\beta-\alpha_1|\leq 3\varep_5/2,\,\rho'\in(0,\delta')\}\subseteq \Omega^1$, provided that $\delta'$ is sufficienly small. 

The open segments $(z(\beta);z(\beta)+de)$, $|\beta-\alpha_1|\leq \varep_5$, cannot intersect the curve $\Gamma_t$, in view of the definition \eqref{ko17} of the points $\alpha_1,\alpha_2$ as a distance-minimizing pair. Therefore, since $\Omega^2$ is connected, these open segments are included in $\Omega^2$. Similarly, the open segments $(z(\beta);z(\beta)-\varep_5e)$, $|\beta-\alpha_1|\leq \varep_5$, cannot intersect the curve $\Gamma_t$ in view of Lemma \ref{chord-arc2}. Therefore, since $\Omega^1$ is connected, these open segments are included in $\Omega^1$. These arguments show that
\begin{equation*}
\begin{split}
&\big\{z(\beta)-\rho e:|\beta-\alpha_1|\leq \varep_5,\,\rho\in(0,\varep_5)\big\}\subseteq\Omega^1,\\
&\big\{z(\beta)+\rho e:|\beta-\alpha_1|\leq \varep_5,\,\rho\in(0,d)\big\}\subseteq\Omega^2.
\end{split}
\end{equation*}
The other two inclusions in \eqref{ko71} follow by a similar argument in a neighborhood of the point $\alpha_2$. 

To prove part (ii) we define the rigid motion transformation $R$ such that \eqref{ko74} holds. Then we define the functions $g_1,h_1:(-\varep_5^2,\varep_5^2)\to (-\varep_5,\varep_5)$ such that
\begin{equation*}
z(\alpha_1+\beta)=R(g_1(\beta),h_1(\beta)).
\end{equation*}
Clearly, $g_1(0)=0$ and $h_1(0)=-d/2$. Moreover, using \eqref{zsmooth}, $\|g_1\|_{C^4}+\|h_1\|_{C^4}\lesssim 1$. Using $z_\alpha(\alpha_1)\cdot e=0$ (see \eqref{ko17.2}), it follows that $h'_1(0)=0$. Using $|z_{\alpha}(\alpha_1)|\geq 1/A$ (see \eqref{zsmooth}), it follows that $|g'_1(0)|\approx 1$. Therefore, the function $g_1$ is locally invertible, in a sufficiently small neighborhood of $0$. We can define $\beta_1:=g_1^{-1}$ and $f_1:=h_1\circ g_1^{-1}$ and it follows that
\begin{equation*}
f_1(0)=-d/2,\quad\beta_1(0)=0,\quad f'_1(0)=0,\quad \|f_1\|_{C^4}+\|\beta_1\|_{C^4}\lesssim 1,\quad |\beta'_1(0)|\approx 1.
\end{equation*}
Similarly, we define the functions $g_2,h_2$ such that $z(\alpha_2+\beta)=R(g_2(\beta),h_2(\beta))$, and then define $\beta_2:=g_2^{-1}$ and $f_2:=h_2\circ g_2^{-1}$. As before, it follows that
\begin{equation*}
f_2(0)=d/2,\quad\beta_2(0)=0,\quad f'_2(0)=0,\quad \|f_2\|_{C^4}+\|\beta_2\|_{C^4}\lesssim 1,\quad |\beta'_2(0)|\approx 1.
\end{equation*}
The inequality $f_2(\rho)-f_1(\rho)\geq d$ in the last line of \eqref{ko76} follows from definition \eqref{ko17} of the points $\alpha_1,\alpha_2$ as a distance-minimizing pair. Finally, $-\beta'_1(0)\beta'_2(0)>0$, which is a consequence of \eqref{ko70.5}, therefore $-\beta'_1(\rho)\beta'_2(\rho)\approx 1$  for all $\rho\in(-\varep_5^4,\varep_5^4)$ as desired. This completes the proof of \eqref{ko76}. 

The first two inclusions in \eqref{ko72} are direct consequences of the first two inclusions in \eqref{ko71}. To prove the last conclusion, we notice first that $R\Big(\big\{(0,y):y\in (f_1(0),f_2(0))\big\}\Big)\subseteq\Omega^2$, as a consequence of \eqref{ko70}. Assume, for contradiction, that the last inclusion fails, so there is $\rho_0\in(-\varep_5^4,\varep_5^4)$ and $y_0\in (f_1(\rho_0),f_2(\rho_0))$ such that $R(\rho_0,y_0)\in\Gamma_t$ and, in addition, $R(\rho,y)\in\Omega^2$ whenever $|\rho|<|\rho_0|$ and $y\in (f_1(\rho),f_2(\rho))$ (recall the last two inclusions in \eqref{ko71}). Therefore, there is $\mu_0\in\mathbb{R}$ such that $z(\mu_0)=R(\rho_0,y_0)$ and $z_\alpha(\mu_0)$ is parallel or anti-parallel to the vector $e$. Since $z_\alpha(\alpha_1)\cdot e=z_\alpha(\alpha_1)\cdot e=0$, it follows that $|\alpha_1-\mu_0|\gtrsim 1$. 

We will derive the contradiction by showing that the curve $z$ has to self-intersect, i.e. there are points $\alpha'$ in a small neighborhood of $\alpha_1$ and $\mu'$ in a small neighborhood of $\mu_0$ such that $z(\alpha')=z(\mu')$. More precisely, using Taylor's expansion around the point $\mu_0$,
\begin{equation*}
\big|z(\mu_0+\beta)-z(\mu_0)-\beta z_\alpha(\mu_0)\big|\lesssim \beta^2.
\end{equation*}
Assuming that $z_\alpha(\mu_0)=ke$, $|k|\approx 1$, it follows that $\big|R^{-1}(z(\mu_0+\beta))-(\rho_0,y_0+k\beta)\big|\lesssim\beta^2$. Since $|y_0|\lesssim\rho_0^2\lesssim \varep_5^8$ (using the fact that $f'_1(0)=f'_2(0)=0$) and letting $R^{-1}(z(\mu_0+\beta))=(\rho,y)$, it follows that $|\rho|\lesssim \varepsilon_5^4$ and $|y-k\beta|\lesssim \varep_5^8$ provided that $|\beta|\leq\varep_5^4$. Therefore there is $\beta$ with $|\beta|\in[\varep_5^4/2,\varep_5^4]$ such that $|\rho|\lesssim \varepsilon_5^4$ and $y\in (-\varep_5^3,f_1(\rho))$, in contradiction with the inclusion in the second line of \eqref{ko72}. This completes the proof of the proposition.
\end{proof}

\subsection{The Birkhoff--Rott operator}\label{BiRo} In this subsection we consider the Birkhoff--Rott integral operator
\begin{equation}\label{ko100}
BR(z,\omega)(\alpha,t)=\frac{1}{2\pi}p.v.\int_{\mathbb{R}}\frac{(z(\alpha,t)-z(\beta,t))^\perp}{|z(\alpha,t)-z(\beta,t)|^2}\omega(\beta,t)\,d\beta
\end{equation}
defined in \eqref{ko1}. We assume that we are in the setting of Proposition \ref{chord-arc3}, namely that there is a point $t\in[0,T]$ such that  $\inf_{\alpha,\beta\in\mathbb{R}}F(z)(\alpha,\beta,t)=\varep\leq\varep_4$. We fix a distance-minimizing pair $(\alpha_1,\alpha_2)=(\alpha_1(t),\alpha_2(t))$ and use the notation in Proposition \ref{chord-arc3}. 

We will only be considering points $z$ in an $\varep_5^{40}$ neighborhood of the point $\frac{z(\alpha_1,t)+z(\alpha_2,t)}{2}$. We consider first the Birkhoff--Rott contribution of points far away. More precisely, assume $\psi:\mathbb{R}\to[0,1]$ is a smooth even function supported in the interval $[-1,1]$ and equal to $1$ in the interval $[-1/2,1/2]$, let $\psi_r(\mu):=\psi(\mu/r)$, $r>0$, and define the function $S^\infty_t:R(B(0,\varep_5^{40}))\to\mathbb{R}^2$,
\begin{equation}\label{ko101}
S^\infty_t(z):=\frac{1}{2\pi}p.v.\int_{\mathbb{R}}\frac{(z-z(\beta,t))^\perp}{|z-z(\beta,t)|^2}\omega(\beta,t)[1-\psi_{\varep_5^{10}}(\beta-\alpha_1)-\psi_{\varep_5^{10}}(\beta-\alpha_2)]\,d\beta.
\end{equation} 
Recall that $|z(\beta,t)-(\beta,0)|\lesssim 1$. In view of \eqref{ko75}--\eqref{ko72} and the periodicity and uniform boundedness of $\omega$ (see Lemma \ref{vort}), we have
\begin{equation}\label{ko102}
\|S^\infty_t\|_{C^3(R(B(0,\varep_5^{40})))}\lesssim 1.
\end{equation}

We define also, for $m\in\{1,2\}$,
\begin{equation*}
S^m_t(\alpha):=\frac{1}{2\pi}p.v.\int_{\mathbb{R}}\frac{(z(\alpha,t)-z(\beta,t))^\perp}{|z(\alpha,t)-z(\beta,t)|^2}\omega(\beta,t)\psi_{\varep_5^{10}}(\beta-\alpha_m)\,d\beta.
\end{equation*}
Therefore, letting $\widetilde{R}:\mathbb{R}^2\to\mathbb{R}^2$, $\widetilde{R}(x):=R(x)-R(0)$, and using the formulas \eqref{ko75}, we have
\begin{equation}\label{ko103}
\begin{split}
&S^1_t(\alpha_1+\beta_1(\mu))=\frac{1}{2\pi}p.v.\int_{\mathbb{R}}\frac{\widetilde{R}\big(\mu-\rho,f_1(\mu)-f_1(\rho)\big)^\perp}{|\mu-\rho|^2+|f_1(\mu)-f_1(\rho)|^2}\omega(\alpha_1+\beta_1(\rho))\beta'_1(\rho)\psi_{\varep_5^{10}}(\beta_1(\rho))\,d\rho,\\
&S^1_t(\alpha_2+\beta_2(\mu))=\frac{1}{2\pi}p.v.\int_{\mathbb{R}}\frac{\widetilde{R}\big(\mu-\rho,f_2(\mu)-f_1(\rho)\big)^\perp}{|\mu-\rho|^2+|f_2(\mu)-f_1(\rho)|^2}\omega(\alpha_1+\beta_1(\rho))\beta'_1(\rho)\psi_{\varep_5^{10}}(\beta_1(\rho))\,d\rho,\\
&S^2_t(\alpha_1+\beta_1(\mu))=\frac{1}{2\pi}p.v.\int_{\mathbb{R}}\frac{\widetilde{R}\big(\mu-\rho,f_1(\mu)-f_2(\rho)\big)^\perp}{|\mu-\rho|^2+|f_1(\mu)-f_2(\rho)|^2}\omega(\alpha_2+\beta_2(\rho))\beta'_2(\rho)\psi_{\varep_5^{10}}(\beta_2(\rho))\,d\rho,\\
&S^2_t(\alpha_2+\beta_2(\mu))=\frac{1}{2\pi}p.v.\int_{\mathbb{R}}\frac{\widetilde{R}\big(\mu-\rho,f_2(\mu)-f_2(\rho)\big)^\perp}{|\mu-\rho|^2+|f_2(\mu)-f_2(\rho)|^2}\omega(\alpha_2+\beta_2(\rho))\beta'_2(\rho)\psi_{\varep_5^{10}}(\beta_2(\rho))\,d\rho,
\end{split}
\end{equation}
for any $\mu\in(-\varep_5^{50},\varep_5^{50})$. Moreover,
\begin{equation}\label{ko104}
\frac{z_\alpha(\alpha_1+\beta_1(\mu))}{|z_\alpha(\alpha_1+\beta_1(\mu))|^2}=\frac{\beta'_1(\mu)\widetilde{R}(1,f'_1(\mu))}{1+|f'_1(\mu)|^2},\qquad \frac{z_\alpha(\alpha_2+\beta_2(\mu))}{|z_\alpha(\alpha_2+\beta_2(\mu))|^2}=\frac{\beta'_2(\mu)\widetilde{R}(1,f'_2(\mu))}{1+|f'_2(\mu)|^2}.
\end{equation} 
Let $\omega_1,\omega_2:(-\varep_5^6,\varep_5^6)\to\mathbb{R}$, $W_1,W_2:(-\varep_5^6,\varep_5^6)\to\mathbb{R}^2$,
\begin{equation}\label{ko105}
\begin{split}
&\omega_1(\rho):=\omega(\alpha_1+\beta_1(\rho))\beta'_1(\rho),\qquad\,\,\,\,\,\omega_2(\rho):=\omega(\alpha_2+\beta_2(\rho))\beta'_2(\rho),\\
&W_1(\rho):=\widetilde{R}^{-1}(v^1(\alpha_1+\beta_1(\rho))),\qquad W_2(\rho):=\widetilde{R}^{-1}(v^1(\alpha_2+\beta_2(\rho))).
\end{split}
\end{equation}
Using also the formula $v^1=BR(z,\omega)+\frac{\omega}{2}\frac{z_\alpha}{|z_\alpha|^2}$, see \eqref{ko1}, we derive our main formula
\begin{equation}\label{ko106}
\begin{split}
W_n(\mu)&=\frac{\omega_n(\mu)(1,f'_n(\mu))}{2(1+|f'_n(\mu)|^2)}+\widetilde{R}^{-1}\big(S_t^\infty(z(\alpha_n+\beta_n(\mu)))\big)\\
&+\sum_{m=1}^2\frac{1}{2\pi}p.v.\int_{\mathbb{R}}\frac{\big(-f_n(\mu)+f_m(\rho),\mu-\rho\big)}{|\mu-\rho|^2+|f_n(\mu)-f_m(\rho)|^2}\omega_m(\rho)\psi_{\varep_5^{10}}(\beta_m(\rho))\,d\rho,
\end{split}
\end{equation}
for $n\in\{1,2\}$. Notice that, in view of \eqref{u1smooth} and Lemma \ref{vort},
\begin{equation}\label{ko107}
\|W_1\|_{C^3(-\varep_5^6,\varep_5^6)}+\|W_2\|_{C^3(-\varep_5^6,\varep_5^6)}+\|\omega_1\|_{L^\infty(-\varep_5^6,\varep_5^6)}+\|\omega_2\|_{L^\infty(-\varep_5^6,\varep_5^6)}\lesssim 1.
\end{equation}

Our plan is to use these bounds and the formulas \eqref{ko106} to show that $\big|[W_1(0)]_2-[W_2(0)]_2\big|\lesssim d\log (1/d)$, where $d=|z(\alpha_2,t)-z(\alpha_1,t)|$. This is a consequence of the main technical result, Proposition \ref{TechProp} below. Given this bound, we will then be able to use Gronwall's inequality to complete the proof of the theorem in section \ref{MainThm2}.

\section{The main technical proposition}\label{MainProp}

Assume $\delta\in(0,10^{-2}]$, $A\geq 1$, and $\eps\in[0,\delta^2/100]$. Assume $f,g:(-\delta,\delta)\to\mathbb{R}$ are $C^5$ functions satisfying
\begin{equation}\label{pi1}
\begin{split}
&\|f\|_{C^3}+\|g\|_{C^3}\leq A,\qquad f(0)=f'(0)=g(0)=g'(0)=0,\\
&f(\alpha)+g(\beta)\geq -\eps\qquad\text{ if }\qquad\alpha,\beta\in(-4\sqrt\eps,4\sqrt\eps).
\end{split}
\end{equation}
For any $r>0$ let $I_r:=(-r,r)$. As before, assume $\psi:\mathbb{R}\to[0,1]$ is a smooth even function supported in the interval $[-1,1]$ and equal to $1$ in the interval $[-1/2,1/2]$, and let $\psi_r(\alpha):=\psi(\alpha/r)$.  Given $\omega\in L^2(I_\delta)$ we define, for $\alpha\in[-\sqrt\eps,\sqrt\eps]$,
\begin{equation}\label{pi7}
\begin{split}
&T_{1,f}\omega(\alpha):=\frac{1}{2\pi}p.v.\int_{\mathbb{R}}\frac{f(\alpha)-f(\beta)}{(\alpha-\beta)^2+(f(\alpha)-f(\beta))^2}\omega(\beta)\psi_{\sqrt{\eps}}(\alpha-\beta)\,d\beta,\\
&T_{2,f}\omega(\alpha):=\frac{1}{2\pi}p.v.\int_{\mathbb{R}}\frac{\alpha-\beta}{(\alpha-\beta)^2+(f(\alpha)-f(\beta))^2}\omega(\beta)\psi_{\sqrt{\eps}}(\alpha-\beta)\,d\beta,\\
&T_{3,(f,g)}\omega(\alpha):=\frac{1}{2\pi}\int_{\mathbb{R}}\frac{2\eps+f(\alpha)+g(\beta)}{(\alpha-\beta)^2+(2\eps+f(\alpha)+g(\beta))^2}\omega(\beta)\psi_{\sqrt{\eps}}(\alpha-\beta)\,d\beta,\\
&T_{4,(f,g)}\omega(\alpha):=\frac{1}{2\pi}\int_{\mathbb{R}}\frac{\alpha-\beta}{(\alpha-\beta)^2+(2\eps+f(\alpha)+g(\beta))^2}\omega(\beta)\psi_{\sqrt{\eps}}(\alpha-\beta)\,d\beta,
\end{split}
\end{equation}
and
\begin{equation}\label{pi7.5}
\begin{split}
&H\omega(\alpha):=\frac{1}{2\pi}\int_{\mathbb{R}}\frac{\omega(\beta)}{\alpha-\beta}\frac{(2\eps+f(\alpha)+g(\alpha))^2}{(\alpha-\beta)^2+(2\eps+f(\alpha)+g(\alpha))^2}\psi_{\sqrt{\eps}/2^6}(\alpha)\psi_{\sqrt{\eps}/2^6}(\alpha-\beta)\,d\beta,\\
&M\omega(\alpha):=\frac{1}{2\pi}\int_{\mathbb{R}}\omega(\beta)\frac{2\eps+f(\alpha)+g(\alpha)}{(\alpha-\beta)^2+(2\eps+f(\alpha)+g(\alpha))^2}\psi_{\sqrt{\eps}/2^6}(\alpha)\psi_{\sqrt{\eps}/2^6}(\alpha-\beta)\,d\beta.
\end{split}
\end{equation}
\medskip

The following is our main definition.

\begin{definition}\label{MainDef}
Let
\begin{equation}\label{pi2}
\mathcal{C}_2:=\Big\{\phi\in C^1_0(I_\delta):|\phi^{(k)}(\alpha)|\leq \frac{\eps}{\eps^{2+k}+|\alpha|^{2+k}}\text{ for any }\alpha\in I_{\delta}\text{ and }k\in\{0,1\}\Big\}.
\end{equation}
For any $h\in L^2(I_\delta)$ let
\begin{equation}\label{pi3}
\|h\|_Z:=\sup_{\phi\in\mathcal{C}_2}\Big|\int_{\mathbb{R}}h(\alpha)H^\ast\phi(\alpha)\,d\alpha\Big|,
\end{equation}
where $H^\ast$ is the dual of the operator $H$ defined in \eqref{pi7.5}.
\end{definition}

We start with a simple lemma:

\begin{lemma}\label{lemma1}
(i) For any $h\in L^2(I_\delta)$ we have
\begin{equation}\label{pi10}
\|h\|_Z\lesssim \eps\|h\|_{C^1(I_{\sqrt{\eps}/8})},
\end{equation}
and
\begin{equation}\label{pi11}
\|h\|_Z\lesssim \int_{I_{\sqrt\eps/8}}|h(\beta)|\frac{\eps^2}{\eps^3+|\beta|^3}\,d\beta.
\end{equation}

(ii) There is a constant $A_1=A_1(A,\delta)$ such that 
\begin{equation}\label{pi12}
A_1^{-1}M^\ast\phi\in \mathcal{C}_2 \qquad\text{ for any }\phi\in \mathcal{C}_2.
\end{equation}
Moreover, for any $h\in L^2(I_\delta)$,
\begin{equation}\label{pi13}
\|Mh\|_{L^\infty(I_\delta)}\lesssim \|h\|_{L^\infty(I_{\sqrt\eps})}.
\end{equation}
\end{lemma}

We will prove this lemma in subsection \ref{subs3}.

We establish now various connections between the operators defined in \eqref{pi7} and \eqref{pi7.5} and the $Z$-norm.

\begin{lemma}\label{lemma2}
(i) For any $h\in L^\infty(I_\delta)$ we have
\begin{equation}\label{pi20}
\|T_{1,f}h\|_Z\lesssim \eps\log(1/\eps)\|h\|_{L^\infty(I_{2\sqrt\eps})},
\end{equation}
and
\begin{equation}\label{pi21}
\|(M-T_{3,(f,g)})h\|_Z+\|(M-T_{3,(g,f)})h\|_Z\lesssim \eps\log(1/\eps)\|h\|_{L^\infty(I_{2\sqrt\eps})}.
\end{equation}

(ii) For any $\phi\in\mathcal{C}_2$ supported in $I_{\sqrt\eps}$,
\begin{equation}\label{pi22}
\begin{split}
&\Big|\int_{I_{\sqrt\eps}}(H+T_{4,(g,f)}-T_{2,f})h(\alpha)\phi(\alpha)\,d\alpha\Big|\\
&+\Big|\int_{I_{\sqrt\eps}}(H+T_{4,(f,g)}-T_{2,g})h(\alpha)\phi(\alpha)\,d\alpha\Big|\lesssim \eps\log(1/\eps)\|h\|_{L^\infty(I_{2\sqrt\eps})}.
\end{split}
\end{equation}
\end{lemma}

We will prove this lemma in subsection \ref{subs4}.

The main result in this section is the following proposition:

\begin{proposition}\label{TechProp}
Assume that $F^+,F^-,G,\omega^+,\omega^-,E^+,E^-,E\in L^2(I_\delta)$ satisfy the estimates
\begin{equation}\label{pi25}
\begin{split}
&\|F^+\|_{C^1(I_{2\sqrt\eps})}+\|F^-\|_{C^1(I_{2\sqrt\eps})}+\|G\|_{C^1(I_{2\sqrt\eps})}+\|\omega^+\|_{L^\infty(I_{2\sqrt\eps})}+\|\omega^-\|_{L^\infty(I_{2\sqrt\eps})}\lesssim 1,\\
&\|E^+\|_{Z}+\|E^-\|_Z+\sup_{\phi\in\mathcal{C}_2,\,\mathrm{supp}\,\phi\subseteq I_{\sqrt\eps/8}}\Big|\int_{\mathbb{R}}E\cdot\phi\,d\alpha\Big|\lesssim\eps\log(1/\eps).
\end{split}
\end{equation}
Moreover, assume that the following identities hold in $I_{\sqrt\eps}$:
\begin{equation}\label{pi26}
\begin{split}
F^+&=T_{1,f}\omega^+-T_{3,(f,g)}\omega^--(1/2)\omega^++E^+,\\
F^-&=-T_{3,(g,f)}\omega^++T_{1,g}\omega^-+(1/2)\omega^-+E^-,\\
G&=(T_{4,(g,f)}-T_{2,f})\omega^++(T_{4,(f,g)}-T_{2,g})\omega^-+E.
\end{split}
\end{equation}
Then
\begin{equation}\label{pi27}
|G(0)|\lesssim \eps\log(1/\eps).
\end{equation}
\end{proposition}

The proof of the proposition uses the two lemmas above, and the following commutator estimate:

\begin{lemma}\label{lemma3}
If $\phi\in \mathcal{C}_2$ then
\begin{equation*}
\|(H^\ast M^\ast-M^\ast H^\ast)\phi\|_{L^1}\lesssim \eps\log(1/\eps).
\end{equation*}
\end{lemma}

We will prove this lemma in subsection \ref{subs5}. 

We show first how to use Lemmas \ref{lemma1}, \ref{lemma2}, and \ref{lemma3} to complete the proof of Proposition \ref{TechProp}.

\begin{proof}[Proof of Proposition \ref{TechProp}] Using Lemma \ref{lemma1} (i) and Lemma \ref{lemma2} (i), the first two equations in \eqref{pi26} together with the assumptions \eqref{pi25} show that
\begin{equation*}
\|\omega^++2M\omega^-\|_Z+\|-2M\omega^++\omega^-\|_Z\lesssim \eps\log(1/\eps).
\end{equation*}
It follows from \eqref{pi12} that $\|M\|_{Z\to Z}\lesssim 1$. Therefore, it follows from the estimate above that
\begin{equation*}
\|(I+4M^2)\omega^+\|_Z+\|(I+4M^2)\omega^-\|_Z\lesssim \eps\log(1/\eps).
\end{equation*}
Therefore, for any $\phi\in\mathcal{C}_2$,
\begin{equation*}
\Big|\int_{\mathbb{R}}(H\omega^++4HM^2\omega^+)\cdot\phi\,d\alpha\Big|+\Big|\int_{\mathbb{R}}(H\omega^-+4HM^2\omega^-)\cdot\phi\,d\alpha\Big|\lesssim \eps\log(1/\eps).
\end{equation*}
Moreover, using \eqref{pi12}, \eqref{pi13}, Lemma \ref{lemma3}, and the assumption $\|\omega^+\|_{L^\infty}\lesssim 1$,
\begin{equation*}
\begin{split}
\Big|\int_{\mathbb{R}}(HM^2\omega^+&-M^2H\omega^+)\cdot\phi\,d\alpha\Big|\lesssim \Big|\int_{\mathbb{R}}(HM-MH)M\omega^+\cdot\phi\,d\alpha\Big|+\Big|\int_{\mathbb{R}}M(HM-MH)\omega^+\cdot\phi\,d\alpha\Big|\\
&\lesssim \Big|\int_{\mathbb{R}}M\omega^+\cdot(M^\ast H^\ast-H^\ast M^\ast)\phi\,d\alpha\Big|+\Big|\int_{\mathbb{R}}\omega^+\cdot(M^\ast H^\ast-H^\ast M^\ast)M^\ast\phi\,d\alpha\Big|\\
&\lesssim \eps\log(1/\eps).
\end{split}
\end{equation*}
Similarly,
\begin{equation*}
\Big|\int_{\mathbb{R}}(HM^2\omega^--M^2H\omega^-)\cdot\phi\,d\alpha\Big|\lesssim \eps\log(1/\eps).
\end{equation*}
Therefore, the last three inequalities show that, for any $\phi\in\mathcal{C}_2$,
\begin{equation}\label{pi30}
\Big|\int_{\mathbb{R}}H\omega^+\cdot[I+4(M^\ast)^2]\phi\,d\alpha\Big|+\Big|\int_{\mathbb{R}}H\omega^-\cdot[I+4(M^\ast)^2]\phi\,d\alpha\Big|\lesssim \eps\log(1/\eps).
\end{equation}

We examine now the last equation in \eqref{pi26} and Lemma \ref{lemma2} (ii). It follows that
\begin{equation*}
\begin{split}
\Big|&\int_{\mathbb{R}}G\cdot[I+4(M^\ast)^2]\phi\,d\alpha\Big|\\
&\leq \Big|\int_{\mathbb{R}}[(T_{4,(g,f)}-T_{2,f})\omega^++(T_{4,(f,g)}-T_{2,g})\omega^-]\cdot[I+4(M^\ast)^2]\phi\,d\alpha\Big|+\Big|\int_{\mathbb{R}}E\cdot[I+4(M^\ast)^2]\phi\,d\alpha\Big|\\
&\lesssim \Big|\int_{\mathbb{R}}(H\omega^++H\omega^-)\cdot[I+4(M^\ast)^2]\phi\,d\alpha\Big|+\eps\log(1/\eps)\\
&\lesssim \eps\log(1/\eps),
\end{split}
\end{equation*}
for any $\phi\in\mathcal{C}_2$ supported in $I_{\sqrt\eps/8}$. Since $|G(\alpha)-G(0)|\lesssim|\alpha|$ it follows that
\begin{equation*}
|G(0)|\Big|\int_{\mathbb{R}}[I+4(M^\ast)^2]\phi\,d\alpha\Big|\lesssim\Big|\int_{\mathbb{R}}G\cdot[I+4(M^\ast)^2]\phi\,d\alpha\Big|+\int_{\mathbb{R}}|\alpha|\big|[I+4(M^\ast)^2]\phi\big|\,d\alpha\lesssim \eps\log(1/\eps),
\end{equation*}
for any $\phi\in\mathcal{C}_2$. We choose now $\phi_0:=(C\eps)^{-1}\psi_\eps\in\mathcal{C}_2$, for a constant $C$ sufficiently large and notice that $(M^\ast)^2\phi\geq 0$. Therefore
\begin{equation*}
\Big|\int_{\mathbb{R}}[I+4(M^\ast)^2]\phi_0\,d\alpha\gtrsim 1,
\end{equation*}
and the desired bound \eqref{pi27} follows.
\end{proof}

\subsection{Proof of Lemma \ref{lemma1}}\label{subs3} In this subsection we prove Lemma \ref{lemma1}. For \eqref{pi10} it suffices to prove that
\begin{equation*}
\Big|\int_{\mathbb{R}}Hh(\alpha)\cdot\phi(\alpha)\,d\alpha\Big|\lesssim\eps\|h\|_{C^1(I_{\sqrt{\eps}/8})},
\end{equation*}
provided that $\phi\in \mathcal{C}_2$. Since
\begin{equation*}
\begin{split}
\big|Hh(\alpha)\big|&=\frac{1}{2\pi}\Big|\int_{\mathbb{R}}\frac{h(\beta)-h(\alpha)}{\alpha-\beta}\frac{(2\eps+f(\alpha)+g(\alpha))^2}{(\alpha-\beta)^2+(2\eps+f(\alpha)+g(\alpha))^2}\psi_{\sqrt{\eps}/2^6}(\alpha)\psi_{\sqrt{\eps}/2^6}(\alpha-\beta)\,d\beta\Big|\\
&\lesssim \|h\|_{C^1(I_{\sqrt{\eps}/8})}(2\eps+f(\alpha)+g(\alpha))\psi_{\sqrt{\eps}/2^6}(\alpha)\\
&\lesssim \eps\psi_{\sqrt{\eps}/2^6}(\alpha)\|h\|_{C^1(I_{\sqrt{\eps}/8})},
\end{split}
\end{equation*}
we have
\begin{equation*}
\Big|\int_{\mathbb{R}}Hh(\alpha)\cdot\phi(\alpha)\,d\alpha\Big|\lesssim\eps\|h\|_{C^1(I_{\sqrt{\eps}/8})}\int_{\mathbb{R}}\psi_{\sqrt{\eps}/2^6}(\alpha)|\phi(\alpha)|\,d\alpha\lesssim \eps\|h\|_{C^1(I_{\sqrt{\eps}/8})},
\end{equation*}
as desired.
\medskip

For \eqref{pi11} it suffices to prove that
\begin{equation}\label{pi110}
|H^\ast\phi(\beta)|\lesssim\frac{\eps^2}{\eps^3+|\beta|^3},\qquad\beta\in I_\delta,
\end{equation}
for any $\phi\in\mathcal{C}_2$. To prove this we write
\begin{equation*}
\begin{split}
H^\ast\phi(\beta)&=\frac{1}{2\pi}\int_{\mathbb{R}}\frac{\phi(\beta+\rho)}{\rho}\frac{(2\eps+f(\beta+\rho)+g(\beta+\rho))^2}{\rho^2+(2\eps+f(\beta+\rho)+g(\beta+\rho))^2}\psi_{\sqrt{\eps}/2^6}(\beta+\rho)\psi_{\sqrt{\eps}/2^6}(\rho)\,d\rho\\
&=H^\ast_{\leq\kappa}\phi(\beta)+H^\ast_{\geq\kappa}\phi(\beta),
\end{split}
\end{equation*}
where, with $X(\mu,\nu):=2\eps+f(\mu)+g(\nu)$,
\begin{equation*}
\begin{split}
&H^\ast_{\leq\kappa}\phi(\beta)=\frac{1}{2\pi}\int_{\mathbb{R}}\frac{\phi(\beta+\rho)}{\rho}\frac{X(\beta+\rho,\beta+\rho)^2}{\rho^2+X(\beta+\rho,\beta+\rho)^2}\psi_{\sqrt{\eps}/2^6}(\beta+\rho)\psi_{\sqrt{\eps}/2^6}(\rho)\psi_\kappa(\rho)\,d\rho,\\
&H^\ast_{\geq\kappa}\phi(\beta)=\frac{1}{2\pi}\int_{\mathbb{R}}\frac{\phi(\beta+\rho)}{\rho}\frac{X(\beta+\rho,\beta+\rho)^2}{\rho^2+X(\beta+\rho,\beta+\rho)^2}\psi_{\sqrt{\eps}/2^6}(\beta+\rho)\psi_{\sqrt{\eps}/2^6}(\rho)(1-\psi_\kappa(\rho))\,d\rho.
\end{split}
\end{equation*}

In proving \eqref{pi110} we may assume $|\beta|\leq\sqrt\eps$. We set $\kappa:=(|\beta|+\eps)/10$ and estimate
\begin{equation}\label{pi111}
\begin{split}
|H^\ast_{\geq\kappa}\phi(\beta)|&\lesssim \int_{\mathbb{R}}\frac{|\phi(\beta+\rho)|}{|\rho|}\frac{\eps^2}{\rho^2}\psi_{\sqrt{\eps}/2^6}(\rho)(1-\psi_\kappa(\rho))\,d\rho\lesssim \int_{\mathbb{R}}\frac{\eps}{\eps^2+(\beta+\rho)^2}\frac{\eps^2}{\kappa^3+|\rho|^3}\,d\rho\lesssim\frac{\eps^2}{\kappa^3}.
\end{split}
\end{equation}
To estimate $|H^\ast_{\leq\kappa}\phi(\beta)|$ we notice that, as a consequence of \eqref{pi1},
\begin{equation*}
\Big|\frac{X(\beta+\rho,\beta+\rho)^2}{\rho^2+X(\beta+\rho,\beta+\rho)^2}\psi_{\sqrt{\eps}/2^6}(\beta+\rho)-\frac{X(\beta,\beta)^2}{\rho^2+X(\beta,\beta)^2}\psi_{\sqrt{\eps}/2^6}(\beta)\Big|\lesssim \frac{\eps|\rho|(\eps+|\beta|)}{\rho^2+\eps^2},
\end{equation*}
if $|\beta|\leq \sqrt\eps$ and $|\rho|\leq\kappa$. Therefore
\begin{equation*}
|H^\ast_{\leq\kappa}\phi(\beta)|\lesssim I+II,
\end{equation*}
where
\begin{equation*}
\begin{split}
I:&=\Big|\int_{\mathbb{R}}\frac{\phi(\beta+\rho)}{\rho}\frac{(2\eps+f(\beta)+g(\beta))^2}{\rho^2+(2\eps+f(\beta)+g(\beta))^2}\psi_{\sqrt{\eps}/2^6}(\beta)\psi_{\sqrt{\eps}/2^6}(\rho)\psi_\kappa(\rho)\,d\rho\Big|\\
&\lesssim\int_{\mathbb{R}}\frac{|\phi(\beta+\rho)-\phi(\beta)|}{|\rho|}\frac{(2\eps+f(\beta)+g(\beta))^2}{\rho^2+(2\eps+f(\beta)+g(\beta))^2}\psi_\kappa(\rho)\,d\rho\\
&\lesssim \int_{\mathbb{R}}\frac{\eps}{\eps^3+|\beta|^3}\frac{\eps^2}{\rho^2+\eps^2}\psi_\kappa(\rho)\,d\rho\\
&\lesssim \frac{\eps^2}{\eps^3+|\beta|^3},
\end{split}
\end{equation*}
and
\begin{equation*}
II:=\int_{\mathbb{R}}\frac{|\phi(\beta+\rho)|}{|\rho|}\frac{\eps|\rho|(\eps+|\beta|)}{\rho^2+\eps^2}\psi_{\sqrt{\eps}/2^6}(\rho)\psi_\kappa(\rho)\,d\rho\lesssim\int_{\mathbb{R}}\frac{\eps}{\eps^2+\beta^2}\frac{\eps(\eps+|\beta|)}{\rho^2+\eps^2}\psi_\kappa(\rho)\,d\rho\lesssim \frac{\eps}{\eps+|\beta|}.
\end{equation*}
Therefore $|H^\ast_{\leq\kappa}\phi(\beta)|\lesssim\eps^2\kappa^{-3}$, and the desired bound \eqref{pi110} follows using also \eqref{pi111}.
\medskip

We prove now the bound \eqref{pi13}. In view of the definition \eqref{pi7.5},
\begin{equation*}
Mh(\alpha)=\frac{1}{2\pi}\int_{\mathbb{R}}h(\alpha-\rho)\frac{2\eps+f(\alpha)+g(\alpha)}{\rho^2+(2\eps+f(\alpha)+g(\alpha))^2}\psi_{\sqrt{\eps}/2^6}(\alpha)\psi_{\sqrt{\eps}/2^6}(\rho)\,d\rho.
\end{equation*}
Therefore, for any $\alpha\in I_\delta$,
\begin{equation*}
|Mh(\alpha)|\lesssim \int_{\mathbb{R}}|h(\alpha-\rho)|\frac{\eps}{\rho^2+\eps^2}\psi_{\sqrt{\eps}/2^6}(\alpha)\psi_{\sqrt{\eps}/2^6}(\rho)\,d\rho\lesssim \|h\|_{L^\infty(I_{\sqrt\eps})},
\end{equation*}
as desired.
\medskip

Finally, to prove \eqref{pi12} it suffices to prove that
\begin{equation}\label{pi120}
|M^\ast\phi(\beta)|\lesssim \frac{\eps}{\eps^2+\beta^2}\qquad\text{ and }\qquad |(M^\ast\phi)'(\beta)|\lesssim \frac{\eps}{\eps^3+|\beta|^3},
\end{equation}
provided that $\phi\in\mathcal{C}_2$ and $|\beta|\leq\sqrt\eps$. Notice that
\begin{equation*}
M^\ast\phi(\beta)=\frac{1}{2\pi}\int_{\mathbb{R}}\phi(\alpha)\frac{2\eps+f(\alpha)+g(\alpha)}{(\alpha-\beta)^2+(2\eps+f(\alpha)+g(\alpha))^2}\psi_{\sqrt{\eps}/2^6}(\alpha)\psi_{\sqrt{\eps}/2^6}(\alpha-\beta)\,d\alpha.
\end{equation*}
Therefore
\begin{equation*}
|M^\ast\phi(\beta)|\lesssim\int_{\mathbb{R}}\frac{\eps}{\eps^2+\alpha^2}\frac{\eps}{(\alpha-\beta)^2+\eps^2}\psi_{\sqrt{\eps}/2^6}(\alpha)\,d\alpha\lesssim \frac{\eps}{\eps^2+\beta^2},
\end{equation*}
as desired.
Moreover, if $|\beta|\leq 10\eps$ then
\begin{equation*}
|(M^\ast\phi)'(\beta)|\lesssim\int_{\mathbb{R}}\frac{\eps}{\eps^2+\alpha^2}\psi_{\sqrt{\eps}/2^6}(\alpha)\Big[\frac{\eps}{(\alpha-\beta)^2+\eps^2}\frac{1}{\sqrt\eps}+\frac{\eps|\alpha-\beta|}{[(\alpha-\beta)^2+\eps^2]^2}\Big]\,d\alpha\lesssim \eps^{-2}.
\end{equation*}

It remains to prove the derivative bound $|(M^\ast\phi)'(\beta)|\lesssim \eps(\eps+|\beta|)^{-3}$ for $|\beta|\geq 10\eps$. We write
\begin{equation*}
\begin{split}
&M^\ast\phi(\beta)=M_1(\beta)+M_2(\beta),\\
&M_1(\beta):=\frac{1}{2\pi}\int_{\mathbb{R}}\phi(\alpha)\frac{\widetilde{X}(\alpha)}{(\alpha-\beta)^2+\widetilde{X}(\alpha)^2}\psi_{\sqrt{\eps}/2^6}(\alpha)\psi_{\sqrt{\eps}/2^6}(\alpha-\beta)[1-\psi(10(\alpha-\beta)/\beta)]\,d\alpha,\\
&M_2(\beta):=\frac{1}{2\pi}\int_{\mathbb{R}}\phi(\beta+\rho)\frac{\widetilde{X}(\beta+\rho)}{\rho^2+\widetilde{X}(\beta+\rho)^2}\psi_{\sqrt{\eps}/2^6}(\beta+\rho)\psi_{\sqrt{\eps}/2^6}(\rho)\psi(10\rho/\beta)\,d\rho.
\end{split}
\end{equation*}
where $\widetilde{X}(\alpha):=X(\alpha,\alpha)=2\eps+f(\alpha)+g(\alpha)$. As before, we estimate
\begin{equation*}
|M'_1(\beta)|\lesssim \int_{\mathbb{R}}\frac{\eps}{\eps^2+\alpha^2}\psi_{\sqrt{\eps}/2^6}(\alpha)\cdot\eps\beta^{-3}\,d\alpha\lesssim \eps\beta^{-3}.
\end{equation*}
On the other hand, using the assumption \eqref{pi1},
\begin{equation*}
|M'_2(\beta)|\lesssim \int_{\mathbb{R}}|\phi'(\beta+\rho)|\frac{\eps}{\rho^2+\eps^2}\psi(10\rho/\beta)\,d\rho+\int_{\mathbb{R}}|\phi(\beta+\rho)|\frac{\eps}{\rho^2+\eps^2}\psi(4\rho/\beta)\beta^{-1}\,d\rho\lesssim \eps\beta^{-3}.
\end{equation*}
This completes the proof of the bound \eqref{pi120}.
\medskip

\subsection{Proof of Lemma \ref{lemma2}}\label{subs4} In this subsection we prove Lemma \ref{lemma2}. For \eqref{pi20} we write, for $\alpha,\beta\in I_{2\sqrt\eps}$,
\begin{equation}\label{pi130}
\frac{f(\alpha)-f(\beta)}{(\alpha-\beta)^2+(f(\alpha)-f(\beta))^2}=\frac{f'(\beta)}{(\alpha-\beta)(1+f'(\beta)^2)}+K(\alpha,\beta),
\end{equation}
where
\begin{equation}\label{pi131}
\begin{split}
K(\alpha,\beta)&:=\frac{f(\alpha)-f(\beta)}{(\alpha-\beta)^2+(f(\alpha)-f(\beta))^2}-\frac{f'(\beta)}{(\alpha-\beta)(1+f'(\beta)^2)}=K_1(\alpha,\beta)+K_2(\alpha,\beta),\\
K_1(\alpha,\beta)&:=\frac{f(\alpha)-f(\beta)-(\alpha-\beta)f'(\beta)}{(\alpha-\beta)^2(1+\widetilde{f}(\alpha,\beta)^2)},\\
K_2(\alpha,\beta)&:=\frac{f'(\beta)[f'(\beta)+\widetilde{f}(\alpha,\beta)][f'(\beta)-\widetilde{f}(\alpha,\beta)]}{(\alpha-\beta)(1+\widetilde{f}(\alpha,\beta)^2)(1+f'(\beta)^2)},\\
\widetilde{f}(\alpha,\beta)&:=\frac{f(\alpha)-f(\beta)}{\alpha-\beta}.
\end{split}
\end{equation}
Since
\begin{equation*}
\begin{split}
&f(\alpha)-f(\beta)=(\alpha-\beta)\int_0^1f'(\beta+(\alpha-\beta)s)\,ds,\\
&f(\alpha)-f(\beta)-(\alpha-\beta)f'(\beta)=(\alpha-\beta)^2\int_0^1(1-s)f''(\beta+(\alpha-\beta)s)\,ds,
\end{split}
\end{equation*}
it follows from \eqref{pi1} that
\begin{equation*}
\sup_{\alpha,\beta\in I_{2\sqrt\eps}}\big[|K(\alpha,\beta)|+|\partial_\alpha K(\alpha,\beta)|+|\partial_\beta K(\alpha,\beta)|\big]\lesssim 1.
\end{equation*}
Therefore, letting
\begin{equation*}
T_{1,f}^1h(\alpha):=\frac{1}{2\pi}\int_{\mathbb{R}}K(\alpha,\beta)h(\beta)\psi_{\sqrt{\eps}}(\alpha-\beta)\,d\beta,
\end{equation*}
and using also \eqref{pi10}, we have
\begin{equation}\label{pi132}
\|T_{1,f}^1h\|_Z\lesssim \eps\|T_{1,f}^1h\|_{C^1(I_{\sqrt{\eps}})}\lesssim\eps.
\end{equation}

To continue, we fix a suitable partition of $\psi$: let $j_0$ denote the largest integer with the property that $2^{j_0}\leq \eps$ and let
\begin{equation}\label{pi100}
\eta_j(\alpha):=
\begin{cases}
\psi_{2^{j_0}}(\alpha)\qquad&\text{ if }j=j_0;\\
\psi_{2^{j}}(\alpha)-\psi_{2^{j-1}}(\alpha)\qquad&\text{ if }j\in[j_0+1,0]\cap\mathbb{Z}.
\end{cases}
\end{equation}
Notice that
\begin{equation}\label{pi100.2}
\sum_{j=j_0}^0\eta_j=\psi.
\end{equation}
For $k_1,k_2\in[j_0,j_0/2+4]\cap\mathbb{Z}$ we define
\begin{equation*}
h_{k_1}(\beta):=h(\beta)\eta_{k_1}(\beta),\qquad H_{k_1,k_2}(\alpha):=\eta_{k_2}(\alpha)\frac{1}{2\pi}\int_{\mathbb{R}}\frac{f'(\beta)}{(\alpha-\beta)(1+f'(\beta)^2)}h_{k_1}(\beta)\psi_{\sqrt{\eps}}(\alpha-\beta)\,d\beta.
\end{equation*}

To prove \eqref{pi20} we may assume that $\|h\|_{I_{2\sqrt\eps}}\leq 1$. In view of \eqref{pi132} it suffices to show that
\begin{equation}\label{pi133}
\Big\|\sum_{k_1,k_2\in[j_0,j_0/2+4]\cap\mathbb{Z}}H_{k_1,k_2}\Big\|_Z\lesssim \eps\log(1/\eps).
\end{equation}
Using \eqref{pi11}, the $L^2$ boundedness of the Hilbert transform, and the assumption $|f'(\beta)|\lesssim\beta$ (see \eqref{pi1}), we estimate
\begin{equation*}
\|H_{k_1,k_2}\|_Z\lesssim \eps^22^{-3k_2}\|H_{k_1,k_2}\|_{L^1}\lesssim \eps^22^{-5k_2/2}\|H_{k_1,k_2}\|_{L^2}\lesssim \eps^22^{-5k_2/2}\cdot 2^{k_1}\|h_{k_1}\|_{L^2}\lesssim \eps^22^{-5k_2/2}2^{3k_1/2}.
\end{equation*}
Therefore
\begin{equation*}
\sum_{k_1,k_2\in[j_0,j_0/2+4]\cap\mathbb{Z},\,k_2\geq k_1-10}\|H_{k_1,k_2}\|_Z\lesssim \eps.
\end{equation*}
On the other hand, using the definition and the assumption $|f'(\beta)|\lesssim\beta$, for any $k_1\in[j_0,j_0/2+4]\cap\mathbb{Z}$ we have
\begin{equation*}
\Big\|\sum_{k_2\in[j_0,k_1-9]\cap\mathbb{Z}}H_{k_1,k_2}\Big\|_{C^1}\lesssim 1.
\end{equation*}
Therefore, using \eqref{pi10} and the last two estimates, for any $k_1\in[j_0,j_0/2+4]\cap\mathbb{Z}$ we have
\begin{equation*}
\Big\|\sum_{k_2\in[j_0,j_0/2+4]\cap\mathbb{Z}}H_{k_1,k_2}\Big\|_{Z}\lesssim \eps,
\end{equation*}
and the desired bound \eqref{pi133} follows.
\medskip

We prove now the bound \eqref{pi21}. We may assume $\|h\|_{L^\infty(I_{2\sqrt\eps})}\leq 1$. By symmetry, it suffices to prove that
\begin{equation}\label{pi135}
\|(M-T_{3,(f,g)})h\|_Z\lesssim \eps\log(1/\eps).
\end{equation}
For $k_1,k_2\in[j_0,j_0/2+4]\cap\mathbb{Z}$ let, as before,
\begin{equation*}
\begin{split}
&h_{k_1}(\beta):=h(\beta)\eta_{k_1}(\beta),\\ &I_{k_1,k_2}(\alpha):=\eta_{k_2}(\alpha)\frac{1}{2\pi}\int_{\mathbb{R}}\frac{2\eps+f(\alpha)+g(\beta)}{(\alpha-\beta)^2+(2\eps+f(\alpha)+g(\beta))^2}h_{k_1}(\beta)\psi_{\sqrt{\eps}}(\alpha-\beta)\,d\beta,\\ 
&\widetilde{I}_{k_1,k_2}(\alpha):=\eta_{k_2}(\alpha)\frac{1}{2\pi}\int_{\mathbb{R}}\frac{2\eps+f(\alpha)+g(\alpha)}{(\alpha-\beta)^2+(2\eps+f(\alpha)+g(\alpha))^2}h_{k_1}(\beta)\psi_{\sqrt{\eps}/2^6}(\alpha)\psi_{\sqrt{\eps}/2^6}(\alpha-\beta)\,d\beta.
\end{split}
\end{equation*}
For \eqref{pi135} it suffices to prove that
\begin{equation}\label{pi136}
\Big\|\sum_{k_1,k_2\in[j_0,j_0/2+4]\cap\mathbb{Z}}(I_{k_1,k_2}-\widetilde{I}_{k_1,k_2})\Big\|_Z\lesssim \eps\log(1/\eps).
\end{equation}

If $k_2\geq j_0/2-20$ we estimate, using \eqref{pi11},
\begin{equation*}
\begin{split}
\Big\|\sum_{k_1\in[j_0,j_0/2+4]\cap\mathbb{Z}}(I_{k_1,k_2}-\widetilde{I}_{k_1,k_2})\Big\|_Z&\lesssim \eps^22^{-3k_2}\Big\|\sum_{k_1\in[j_0,j_0/2+4]\cap\mathbb{Z}}[|I_{k_1,k_2}|+|\widetilde{I}_{k_1,k_2}|]\Big\|_{L^1}\lesssim\eps^22^{-3k_2}2^{k_2}\lesssim\eps.
\end{split}
\end{equation*}
On the other hand, if $k_1\geq j_0/2-10$ we estimate, using \eqref{pi10},
\begin{equation*}
\begin{split}
\Big\|\sum_{k_2\in[j_0,j_0/2-21]\cap\mathbb{Z}}(I_{k_1,k_2}-\widetilde{I}_{k_1,k_2})\Big\|_Z&\lesssim \eps\Big\|\sum_{k_2\in[j_0,j_0/2-21]\cap\mathbb{Z}}(I_{k_1,k_2}-\widetilde{I}_{k_1,k_2})\Big\|_{C^1}\lesssim\eps.
\end{split}
\end{equation*}
Therefore, for \eqref{pi136} it suffices to prove that, for any $k_1\in[j_0,j_0/2-11]\cap\mathbb{Z}$ fixed,
\begin{equation}\label{pi137}
\Big\|\sum_{k_2\in[j_0,j_0/2-21]\cap\mathbb{Z}}(I_{k_1,k_2}-\widetilde{I}_{k_1,k_2})\Big\|_Z\lesssim \eps.
\end{equation}

Notice that
\begin{equation}\label{pi137.5}
(I_{k_1,k_2}-\widetilde{I}_{k_1,k_2})(\alpha)=\eta_{k_2}(\alpha)\frac{1}{2\pi}\int_{\mathbb{R}}\frac{[g(\beta)-g(\alpha)][(\alpha-\beta)^2-X(\alpha,\beta)X(\alpha,\alpha)]}{[(\alpha-\beta)^2+X(\alpha,\beta)^2][(\alpha-\beta)^2+X(\alpha,\alpha)^2]}h_{k_1}(\beta)\,d\beta
\end{equation}
if $k_1\in[j_0,j_0/2-11]$ and $k_2\in[j_0,j_0/2-21]\cap\mathbb{Z}$, where, as before,
\begin{equation*}
X(\alpha,\beta)=2\eps+f(\alpha)+g(\beta).
\end{equation*}
Notice that $X(\alpha,\beta)\approx\eps$, $(\partial_\alpha X)(\alpha,\beta)\lesssim |\alpha|$, $(\partial_\beta X)(\alpha,\beta)\lesssim |\beta|$, see \eqref{pi1}. Therefore, using first \eqref{pi11}, we can estimate if $k_2\geq k_1-10$, 
\begin{equation*}
\|(I_{k_1,k_2}-\widetilde{I}_{k_1,k_2})\|_Z\lesssim \eps^22^{-3k_2}\cdot 2^{k_2}\|(I_{k_1,k_2}-\widetilde{I}_{k_1,k_2})\|_{L^\infty}\lesssim \eps^22^{-2k_2}\cdot 2^{k_2}\log(1+2^{k_2}/\eps)
\end{equation*}
Therefore
\begin{equation}\label{pi138}
\Big\|\sum_{k_2\in[j_0,j_0/2-21]\cap\mathbb{Z},\,k_2\geq k_1-10}(I_{k_1,k_2}-\widetilde{I}_{k_1,k_2})\Big\|_Z\lesssim \eps.
\end{equation}
On the other hand, using \eqref{pi1} and the definition \eqref{pi137.5},
\begin{equation*}
\Big\|\sum_{k_2\in[j_0,k_1-11]\cap\mathbb{Z}}(I_{k_1,k_2}-\widetilde{I}_{k_1,k_2})\Big\|_{C^1}\lesssim 1.
\end{equation*}
The desired bound \eqref{pi137} follows using also \eqref{pi10} and \eqref{pi138}.
\medskip

We prove now the bound \eqref{pi22}. We may assume $\|h\|_{L^\infty(I_{2\sqrt\eps})}\leq 1$. By symmetry, it suffices to prove that
\begin{equation}\label{pi140}
\Big|\int_{I_{\sqrt\eps}}(H+T_{4,(g,f)}-T_{2,f})h(\alpha)\phi(\alpha)\,d\alpha\Big|\lesssim \eps\log(1/\eps).
\end{equation}

Using the definitions, we write
\begin{equation*}
\begin{split}
&[-T_{4,(g,f)}+T_{2,f}]h(\alpha)\\
&=\frac{1}{2\pi}\int_{\mathbb{R}}\frac{(\alpha-\beta)\big[(2\eps+f(\alpha)+g(\alpha))^2+2(2\eps+f(\alpha)+g(\alpha))(f(\beta)-f(\alpha))\big]}{[(\alpha-\beta)^2+(f(\alpha)-f(\beta))^2]\cdot[(\alpha-\beta)^2+(2\eps+f(\beta)+g(\alpha))^2]}h(\beta)\psi_{\sqrt{\eps}}(\alpha-\beta)\,d\beta\\
&=J^1(\alpha)+J^2(\alpha),
\end{split}
\end{equation*}
where, with $X(\mu,\nu):=2\eps+f(\mu)+g(\nu)$ as before,
\begin{equation*}
\begin{split}
&J^1(\alpha):=\frac{1}{2\pi}\int_{\mathbb{R}}\frac{(\alpha-\beta)\big[2(2\eps+f(\alpha)+g(\alpha))(f(\beta)-f(\alpha))\big]}{[(\alpha-\beta)^2+(f(\alpha)-f(\beta))^2]\cdot[(\alpha-\beta)^2+X(\beta,\alpha)^2]}h(\beta)\psi_{\sqrt{\eps}}(\alpha-\beta)\,d\beta,\\
&J^2(\alpha):=\frac{1}{2\pi}\int_{\mathbb{R}}\frac{(\alpha-\beta)(2\eps+f(\alpha)+g(\alpha))^2}{[(\alpha-\beta)^2+(f(\alpha)-f(\beta))^2]\cdot[(\alpha-\beta)^2+X(\beta,\alpha)^2]}h(\beta)\psi_{\sqrt{\eps}}(\alpha-\beta)\,d\beta.
\end{split}
\end{equation*}

Notice that, using \eqref{pi1},
\begin{equation}\label{pi141}
|J^1(\alpha)|\lesssim \int_{\mathbb{R}}\frac{|f(\beta)-f(\alpha)|}{|\beta-\alpha|}\frac{\eps}{\eps^2+(\alpha-\beta)^2}\psi_{\sqrt{\eps}}(\alpha-\beta)\,d\beta\lesssim|\alpha|+\eps.
\end{equation}
Moreover, let
\begin{equation*}
J^3(\alpha):=\frac{1}{2\pi}\int_{\mathbb{R}}\frac{(2\eps+f(\alpha)+g(\alpha))^2}{[1+f'(\alpha)^2]\cdot[(\alpha-\beta)^2+X(\alpha,\alpha)^2]}\frac{h(\beta)}{\alpha-\beta}\psi_{\sqrt{\eps}}(\alpha-\beta)\,d\beta.
\end{equation*}
Recalling the definition $\widetilde{f}(\alpha,\beta)=\frac{f(\alpha)-f(\beta)}{\alpha-\beta}$, we estimate, using \eqref{pi1},
\begin{equation*}
\Big|[1+\widetilde{f}(\alpha,\beta)^2]\cdot [(\alpha-\beta)^2+X(\beta,\alpha)^2]-[1+f'(\alpha)^2]\cdot[(\alpha-\beta)^2+X(\alpha,\alpha)^2]\Big|\lesssim \eps|\alpha-\beta|(|\alpha|+|\beta|),
\end{equation*}
for any $\alpha,\beta\in I_{2\sqrt\eps}$. Therefore
\begin{equation}\label{pi142}
|J^2(\alpha)-J^3(\alpha)|\lesssim\int_{\mathbb{R}}\frac{\eps^2\cdot\eps|\alpha-\beta|(|\alpha|+|\beta|) }{[(\alpha-\beta)^2+\eps^2]^2}\frac{|h(\beta)|}{|\alpha-\beta|}\psi_{\sqrt{\eps}}(\alpha-\beta)\,d\beta\lesssim |\alpha|+\eps.
\end{equation}

It follows from \eqref{pi141} and \eqref{pi142} that
\begin{equation*}
\int_{I_{\sqrt\eps}}|J^1(\alpha)+J^2(\alpha)-J^3(\alpha)|\cdot|\phi(\alpha)|\,d\alpha\lesssim \eps\log(1/\eps)
\end{equation*}
for any $\phi\in\mathcal{C}_2$. Therefore, for \eqref{pi140} it remains to prove that, for any $\phi\in\mathcal{C}_2$ supported in $I_{\sqrt\eps}$,
\begin{equation}\label{pi145}
\Big|\int_{\mathbb{R}}[Hh(\alpha)-J^3(\alpha)]\phi(\alpha)\,d\alpha\Big|\lesssim \eps\log(1/\eps).
\end{equation}

Using the definitions, the left-hand side of \eqref{pi145} is equal to
\begin{equation*}
\Big|\int_{\mathbb{R}}h(\beta)\Phi(\beta)\,d\beta\Big|
\end{equation*}
where
\begin{equation*}
\Phi(\beta):=\frac{1}{2\pi}\int_{\mathbb{R}}\frac{\phi(\beta+\rho)}{\rho}\frac{X(\beta+\rho,\beta+\rho)^2}{\rho^2+X(\beta+\rho,\beta+\rho)^2}\Big[\psi_{\sqrt{\eps}/2^6}(\beta+\rho)\psi_{\sqrt{\eps}/2^6}(\rho)-\frac{\psi_{\sqrt\eps}(\rho)}{1+f'(\beta+\rho)^2}\Big]\,d\rho.
\end{equation*}
An argument similar to the one used in the proof of \eqref{pi110} shows that $|\Phi(\beta)|\lesssim\sqrt\eps\mathbf{1}_{I_{2\sqrt\eps}}(\beta)$, and the desired bound \eqref{pi145} follows.

\subsection{Proof of Lemma \ref{lemma3}}\label{subs5} In this subsection we prove Lemma \ref{lemma3}. In view of the definitions we have
\begin{equation}\label{pi150}
(H^\ast M^\ast-M^\ast H^\ast)\phi(\alpha)=\frac{1}{4\pi^2}\int_{\mathbb{R}}K(\alpha,\beta)\phi(\beta)\,d\beta,
\end{equation}
where, with $\widetilde{X}(\alpha)=X(\alpha,\alpha)=2\eps+f(\alpha)+g(\alpha)$,
\begin{equation*}
\begin{split}
K(\alpha,\beta):&=\int_{\mathbb{R}}\frac{\psi_{\sqrt\eps/2^6}(\mu-\alpha)}{\mu-\alpha}\frac{\widetilde{X}(\mu)^2\psi_{\sqrt\eps/2^6}(\mu)}{(\mu-\alpha)^2+\widetilde{X}(\mu)^2}\cdot\frac{\widetilde{X}(\beta)\psi_{\sqrt\eps/2^6}(\beta)}{(\beta-\mu)^2+\widetilde{X}(\beta)^2}\psi_{\sqrt\eps/2^6}(\beta-\mu)\,d\mu\\
&-\int_{\mathbb{R}}\frac{\widetilde{X}(\mu)\psi_{\sqrt\eps/2^6}(\mu)}{(\mu-\alpha)^2+\widetilde{X}(\mu)^2}\psi_{\sqrt\eps/2^6}(\mu-\alpha)\cdot\frac{\psi_{\sqrt\eps/2^6}(\beta-\mu)}{\beta-\mu}\frac{\widetilde{X}(\beta)^2\psi_{\sqrt\eps/2^6}(\beta)}{(\beta-\mu)^2+\widetilde{X}(\beta)^2}\,d\mu.
\end{split}
\end{equation*}
After changes of variables this becomes
\begin{equation*}
K(\alpha,\beta)=\int_{\mathbb{R}}\frac{\psi_{\sqrt\eps/2^6}(\rho)}{\rho}\psi_{\sqrt\eps/2^6}(\beta-\alpha-\rho)\widetilde{K}(\alpha,\beta;\rho)\,d\rho,
\end{equation*}
where
\begin{equation*}
\begin{split}
\widetilde{K}(\alpha,\beta;\rho):&=\frac{\widetilde{X}(\alpha+\rho)^2\psi_{\sqrt\eps/2^6}(\alpha+\rho)}{\rho^2+\widetilde{X}(\alpha+\rho)^2}\cdot\frac{\widetilde{X}(\beta)\psi_{\sqrt\eps/2^6}(\beta)}{(\beta-\alpha-\rho)^2+\widetilde{X}(\beta)^2}\\
&-\frac{\widetilde{X}(\beta)^2\psi_{\sqrt\eps/2^6}(\beta)}{\rho^2+\widetilde{X}(\beta)^2}\cdot\frac{\widetilde{X}(\beta-\rho)\psi_{\sqrt\eps/2^6}(\beta-\rho)}{(\beta-\alpha-\rho)^2+\widetilde{X}(\beta-\rho)^2}.
\end{split}
\end{equation*}

We will show that, for any $\alpha,\beta\in\mathbb{R}$,
\begin{equation}\label{pi152.5}
|K(\alpha,\beta)|\lesssim \psi_{\sqrt\eps}(\beta)\psi_{\sqrt\eps}(\alpha)\frac{\eps(|\alpha|+|\beta|+\eps)}{(\beta-\alpha)^2+\eps^2}.
\end{equation}
Assuming this, for any $\phi\in\mathcal{C}_2$ we estimate
\begin{equation*}
\Big\|\int_{\mathbb{R}}\phi(\beta)K(\alpha,\beta)\,d\beta\Big\|_{L^1_\alpha}\lesssim \int_{\mathbb{R}}\int_{\mathbb{R}}\psi_{\sqrt\eps}(\beta)\psi_{\sqrt\eps}(\alpha)\frac{\eps(|\alpha|+|\beta|+\eps)}{(\beta-\alpha)^2+\eps^2}\cdot\frac{\eps}{\eps^2+\beta^2}\,d\beta d\alpha\lesssim \eps\log(1/\eps),
\end{equation*}
as desired.

It remains to prove \eqref{pi152.5}. We write
\begin{equation*}
\begin{split}
&\widetilde{K}(\alpha,\beta;\rho)=\widetilde{K}_1(\alpha,\beta;\rho)+\widetilde{K}_2(\alpha,\beta;\rho)+\widetilde{K}_3(\alpha,\beta;\rho),\\
&\widetilde{K}_1(\alpha,\beta;\rho):=\frac{\widetilde{X}(\beta)\psi_{\sqrt\eps/2^6}(\beta)}{(\beta-\alpha-\rho)^2+\widetilde{X}(\beta)^2}\Big[\frac{\widetilde{X}(\alpha+\rho)^2\psi_{\sqrt\eps/2^6}(\alpha+\rho)}{\rho^2+\widetilde{X}(\alpha+\rho)^2}-\frac{\widetilde{X}(\alpha)^2\psi_{\sqrt\eps/2^6}(\alpha)}{\rho^2+\widetilde{X}(\alpha)^2}\Big],\\
&\widetilde{K}_2(\alpha,\beta;\rho):=\frac{\widetilde{X}(\beta)\psi_{\sqrt\eps/2^6}(\beta)}{(\beta-\alpha-\rho)^2+\widetilde{X}(\beta)^2}\Big[\frac{\widetilde{X}(\alpha)^2\psi_{\sqrt\eps/2^6}(\alpha)}{\rho^2+\widetilde{X}(\alpha)^2}-\frac{\widetilde{X}(\beta)^2\psi_{\sqrt\eps/2^6}(\beta)}{\rho^2+\widetilde{X}(\beta)^2}\Big],\\
&\widetilde{K}_3(\alpha,\beta;\rho):=\frac{\widetilde{X}(\beta)^2\psi_{\sqrt\eps/2^6}(\beta)}{\rho^2+\widetilde{X}(\beta)^2}\Big[\frac{\widetilde{X}(\beta)\psi_{\sqrt\eps/2^6}(\beta)}{(\beta-\alpha-\rho)^2+\widetilde{X}(\beta)^2}-\frac{\widetilde{X}(\beta-\rho)\psi_{\sqrt\eps/2^6}(\beta-\rho)}{(\beta-\alpha-\rho)^2+\widetilde{X}(\beta-\rho)^2}\Big].
\end{split}
\end{equation*}
Recalling that $X(\mu)\approx\eps$ and $|X'(\mu)|\lesssim |\mu|+\eps$ if $\mu\in I_{\sqrt\eps}$, we have the simple estimates
\begin{equation}\label{pi153}
\begin{split}
&|\widetilde{K}_1(\alpha,\beta;\rho)|\lesssim \frac{\eps\psi_{\sqrt\eps/2^6}(\beta)}{(\beta-\alpha-\rho)^2+\eps^2}\cdot \frac{|\rho|\eps(|\alpha|+|\rho|+\eps)}{\rho^2+\eps^2},\\
&|\widetilde{K}_3(\alpha,\beta;\rho)|\lesssim \frac{\eps^2\psi_{\sqrt\eps/2^6}(\beta)}{\rho^2+\eps^2}\cdot \frac{|\rho|(|\beta|+|\rho|+\eps)}{(\beta-\alpha-\rho)^2+\eps^2}.
\end{split}
\end{equation}
if $\rho,\alpha-\beta-\rho\in I_{\sqrt\eps}$. Therefore, letting
\begin{equation}\label{pi152}
K_j(\alpha,\beta):=\int_{\mathbb{R}}\frac{\psi_{\sqrt\eps/2^6}(\rho)}{\rho}\psi_{\sqrt\eps/2^6}(\beta-\alpha-\rho)\widetilde{K}_j(\alpha,\beta;\rho)\,d\rho,\qquad j\in\{1,2,3\},
\end{equation}
it follows that
\begin{equation*}
\begin{split}
|K_1(\alpha,\beta)|+|K_3(\alpha,\beta)|&\lesssim \psi_{\sqrt\eps/2^6}(\beta)\int_{\mathbb{R}}\frac{\eps\psi_{\sqrt\eps/2^6}(\rho)}{\rho^2+\eps^2}\frac{\eps\psi_{\sqrt\eps/2^6}(\beta-\alpha-\rho)}{(\beta-\alpha-\rho)^2+\eps^2}(|\alpha|+|\beta|+|\rho|+\eps)\,d\rho\\
&\lesssim \psi_{\sqrt\eps}(\beta)\psi_{\sqrt\eps}(\alpha)\frac{\eps(|\alpha|+|\beta|+\eps)}{(\beta-\alpha)^2+\eps^2},
\end{split}
\end{equation*}
as desired.

For \eqref{pi152.5} it remains to prove that
\begin{equation}\label{pi155}
|K_2(\alpha,\beta)|\lesssim \psi_{\sqrt\eps}(\beta)\psi_{\sqrt\eps}(\alpha)\frac{\eps(|\alpha|+|\beta|+\eps)}{(\beta-\alpha)^2+\eps^2}.
\end{equation}
For this we examine the formula \eqref{pi152} and decompose
\begin{equation*}
K_2(\alpha,\beta)=K_2^1(\alpha,\beta)+K_2^2(\alpha,\beta),
\end{equation*}
where, with $\kappa:=(\eps+|\alpha-\beta|)/100$,
\begin{equation*}
\begin{split}
&K_2^1(\alpha,\beta):=\int_{\mathbb{R}}\frac{\psi_{\kappa}(\rho)}{\rho}\psi_{\sqrt\eps/2^6}(\beta-\alpha-\rho)\widetilde{K}_2(\alpha,\beta;\rho)\,d\rho,\\
&K_2^2(\alpha,\beta):=\int_{\mathbb{R}}\frac{\psi_{\sqrt\eps/2^6}(\rho)-\psi_{\kappa}(\rho)}{\rho}\psi_{\sqrt\eps/2^6}(\beta-\alpha-\rho)\widetilde{K}_2(\alpha,\beta;\rho)\,d\rho.
\end{split}
\end{equation*}
We can bound the contribution of $K_2^2$ as before; as in \eqref{pi153}, we estimate first
\begin{equation*}
|\widetilde{K}_2(\alpha,\beta;\rho)|\lesssim \frac{\eps\psi_{\sqrt\eps/2^6}(\beta)}{(\beta-\alpha-\rho)^2+\eps^2}\cdot \frac{|\alpha-\beta|\eps(|\alpha|+|\beta|+\eps)}{\rho^2+\eps^2}
\end{equation*}
if $\rho,\alpha-\beta-\rho\in I_{\sqrt\eps}$, therefore
\begin{equation}\label{pi156}
\begin{split}
|K_2^2(\alpha,\beta)|&\lesssim \int_{\mathbb{R}}\frac{\psi_{\sqrt\eps/2^6}(\rho)-\psi_{\kappa}(\rho)}{|\rho|}\psi_{\sqrt\eps/2^6}(\beta-\alpha-\rho)\frac{\eps^2\psi_{\sqrt\eps/2^6}(\beta)}{(\beta-\alpha-\rho)^2+\eps^2}\cdot \frac{|\alpha-\beta|(|\alpha|+|\beta|+\eps)}{\rho^2+\eps^2}\,d\rho\\
&\lesssim \eps\psi_{\sqrt\eps}(\beta)\psi_{\sqrt\eps}(\alpha)(|\alpha|+|\beta|+\eps)\int_{\mathbb{R}}\frac{\psi_{\sqrt\eps/2^6}(\rho)-\psi_{\kappa}(\rho)}{|\rho|}\frac{\eps\psi_{\sqrt\eps/2^6}(\beta-\alpha-\rho)}{(\beta-\alpha-\rho)^2+\eps^2}\cdot \frac{|\alpha-\beta|}{\rho^2+\eps^2}\,d\rho\\
&\lesssim \frac{\eps\psi_{\sqrt\eps}(\beta)\psi_{\sqrt\eps}(\alpha)(|\alpha|+|\beta|+\eps)}{(\beta-\alpha)^2+\eps^2}.
\end{split}
\end{equation}

To bound $|K_2^1|$ we notice that
\begin{equation*}
K_2^1(\alpha,\beta)=\int_{\mathbb{R}}\frac{\psi_{\kappa}(\rho)}{\rho}\big[\psi_{\sqrt\eps/2^6}(\beta-\alpha-\rho)\widetilde{K}_2(\alpha,\beta;\rho)-\psi_{\sqrt\eps/2^6}(\beta-\alpha)\widetilde{K}'_2(\alpha,\beta;\rho)\big]\,d\rho,
\end{equation*}
where
\begin{equation*}
\widetilde{K}'_2(\alpha,\beta;\rho):=\frac{\widetilde{X}(\beta)\psi_{\sqrt\eps/2^6}(\beta)}{(\beta-\alpha)^2+\widetilde{X}(\beta)^2}\Big[\frac{\widetilde{X}(\alpha)^2\psi_{\sqrt\eps/2^6}(\alpha)}{\rho^2+\widetilde{X}(\alpha)^2}-\frac{\widetilde{X}(\beta)^2\psi_{\sqrt\eps/2^6}(\beta)}{\rho^2+\widetilde{X}(\beta)^2}\Big].
\end{equation*}
Clearly, $K_2^1(\alpha,\beta)=0$ unless $\alpha,\beta\in I_{\sqrt\eps/8}$. On the other hand, if $\alpha,\beta\in I_{\sqrt\eps/8}$ and $\rho\in I_\kappa$ then we estimate
\begin{equation*}
\begin{split}
\big|\psi_{\sqrt\eps/2^6}&(\beta-\alpha-\rho)\widetilde{K}_2(\alpha,\beta;\rho)-\psi_{\sqrt\eps/2^6}(\beta-\alpha)\widetilde{K}'_2(\alpha,\beta;\rho)\big|\\
&\lesssim\Big|\psi_{\sqrt\eps/2^6}(\beta-\alpha-\rho)\frac{\widetilde{X}(\beta)\psi_{\sqrt\eps/2^6}(\beta)}{(\beta-\alpha-\rho)^2+\widetilde{X}(\beta)^2}-\psi_{\sqrt\eps/2^6}(\beta-\alpha)\frac{\widetilde{X}(\beta)\psi_{\sqrt\eps/2^6}(\beta)}{(\beta-\alpha)^2+\widetilde{X}(\beta)^2}\Big|\\
&\times\Big|\frac{\widetilde{X}(\alpha)^2\psi_{\sqrt\eps/2^6}(\alpha)}{\rho^2+\widetilde{X}(\alpha)^2}-\frac{\widetilde{X}(\beta)^2\psi_{\sqrt\eps/2^6}(\beta)}{\rho^2+\widetilde{X}(\beta)^2}\Big|\\
&\lesssim \frac{\eps|\rho|(|\beta-\alpha|+\eps)}{[(\beta-\alpha)^2+\eps^2]^2}\cdot \frac{|\alpha-\beta|\eps(|\alpha|+|\beta|+\eps)}{\rho^2+\eps^2}.
\end{split}
\end{equation*}
Therefore, if $\alpha,\beta\in I_{\sqrt\eps/8}$,
\begin{equation*}
|K_2^1(\alpha,\beta)|\lesssim \int_{\mathbb{R}}\frac{\psi_{\kappa}(\rho)}{|\rho|}\frac{\eps|\rho|(|\beta-\alpha|+\eps)}{[(\beta-\alpha)^2+\eps^2]^2}\cdot \frac{|\alpha-\beta|\eps(|\alpha|+|\beta|+\eps)}{\rho^2+\eps^2}\,d\rho\lesssim \frac{\eps(|\alpha|+|\beta|+\eps)}{(\beta-\alpha)^2+\eps^2}.
\end{equation*}
The desired bound \eqref{pi155} follows using also \eqref{pi156}, which completes the proof of the lemma.

\section{Proof of the main theorem, II: completion of the proof}\label{MainThm2}

We show first that the difference of the normal components of the velocity field $v^1$ at the points $\alpha_1$ and $\alpha_2$ is very small. More precisely:

\begin{lemma}\label{diffvel}
With the notation in subsection \ref{BiRo}, we have
\begin{equation*}
\big|[W_1(0)]_2-[W_2(0)]_2\big|\lesssim d\log(1/d).
\end{equation*}
\end{lemma}

\begin{proof}[Proof of Lemma \ref{diffvel}] We would like to apply Proposition \ref{TechProp}. Let
\begin{equation*}
\eps:=d/2,\quad g(\rho):=-f_1(\rho)-\eps,\quad f(\rho):=f_2(\rho)-\eps,\quad \omega^-(\rho)=-\omega_1(\rho),\quad \omega^+(\rho):=\omega_2(\rho).
\end{equation*}
Then, using the definitions \eqref{pi7},
\begin{equation}\label{las1}
\begin{split}
T_{1,f}\omega^+(\alpha)&=\frac{1}{2\pi}p.v.\int_{\mathbb{R}}\frac{f_2(\alpha)-f_2(\beta)}{(\alpha-\beta)^2+(f_2(\alpha)-f_2(\beta))^2}\omega_2(\beta)\psi_{\sqrt{\eps}}(\alpha-\beta)\,d\beta,\\
-T_{3,(f,g)}\omega^-(\alpha)&=\frac{1}{2\pi}\int_{\mathbb{R}}\frac{f_2(\alpha)-f_1(\beta)}{(\alpha-\beta)^2+(f_2(\alpha)-f_1(\beta))^2}\omega_1(\beta)\psi_{\sqrt{\eps}}(\alpha-\beta)\,d\beta,\\
T_{1,g}\omega^-(\alpha)&=\frac{1}{2\pi}p.v.\int_{\mathbb{R}}\frac{f_1(\alpha)-f_1(\beta)}{(\alpha-\beta)^2+(f_1(\alpha)-f_1(\beta))^2}\omega_1(\beta)\psi_{\sqrt{\eps}}(\alpha-\beta)\,d\beta,\\
-T_{3,(g,f)}\omega^+(\alpha)&=\frac{1}{2\pi}\int_{\mathbb{R}}\frac{f_1(\alpha)-f_2(\beta)}{(\alpha-\beta)^2+(f_1(\alpha)-f_2(\beta))^2}\omega_2(\beta)\psi_{\sqrt{\eps}}(\alpha-\beta)\,d\beta,\\
\end{split}
\end{equation}
and
\begin{equation}\label{las2}
\begin{split}
T_{4,(g,f)}\omega^+(\alpha)&=\frac{1}{2\pi}\int_{\mathbb{R}}\frac{\alpha-\beta}{(\alpha-\beta)^2+(f_1(\alpha)-f_2(\beta))^2}\omega_2(\beta)\psi_{\sqrt{\eps}}(\alpha-\beta)\,d\beta,\\
T_{2,f}\omega^+(\alpha)&=\frac{1}{2\pi}p.v.\int_{\mathbb{R}}\frac{\alpha-\beta}{(\alpha-\beta)^2+(f_2(\alpha)-f_2(\beta))^2}\omega_2(\beta)\psi_{\sqrt{\eps}}(\alpha-\beta)\,d\beta,\\
T_{4,(f,g)}\omega^-(\alpha)&=-\frac{1}{2\pi}\int_{\mathbb{R}}\frac{\alpha-\beta}{(\alpha-\beta)^2+(f_2(\alpha)-f_1(\beta))^2}\omega_1(\beta)\psi_{\sqrt{\eps}}(\alpha-\beta)\,d\beta,\\
T_{2,g}\omega^-(\alpha)&=-\frac{1}{2\pi}p.v.\int_{\mathbb{R}}\frac{\alpha-\beta}{(\alpha-\beta)^2+(f_1(\alpha)-f_1(\beta))^2}\omega_1(\beta)\psi_{\sqrt{\eps}}(\alpha-\beta)\,d\beta.
\end{split}
\end{equation}

We examine now the formulas \eqref{ko106}
\begin{equation}\label{las2.5}
\begin{split}
W_n(\mu)&=\frac{\omega_n(\mu)(1,f'_n(\mu))}{2(1+|f'_n(\mu)|^2)}+\widetilde{R}^{-1}\big(S_t^\infty(z(\alpha_n+\beta_n(\mu)))\big)\\
&+\sum_{m=1}^2\frac{1}{2\pi}p.v.\int_{\mathbb{R}}\frac{\big(-f_n(\mu)+f_m(\rho),\mu-\rho\big)}{|\mu-\rho|^2+|f_n(\mu)-f_m(\rho)|^2}\omega_m(\rho)\psi_{\varep_5^{10}}(\beta_m(\rho))\,d\rho.
\end{split}
\end{equation}
Observe that
\begin{equation}\label{las3}
\begin{split}
&-[W_1(\mu)]_1=\frac{\omega^-(\mu)}{2}+T_{1,g}\omega^-(\mu)-T_{3,(g,f)}\omega^+(\mu)+E^-(\mu),\\
&-[W_2(\mu)]_1=-\frac{\omega^+(\mu)}{2}-T_{3,(f,g)}\omega^-(\mu)+T_{1,f}\omega^+(\mu)+E^+(\mu),
\end{split}
\end{equation}
where
\begin{equation*}
\begin{split}
&E^-=E^-_1+E^-_2+E^-_3,\\
&E^-_1(\mu):=\frac{\omega_1(\mu)|f'_1(\mu)|^2}{2(1+|f'_1(\mu)|^2)},\\
&E^-_2(\mu):=-\big[\widetilde{R}^{-1}\big(S_t^\infty(z(\alpha_1+\beta_1(\mu)))\big)\big]_1,\\
&E^-_3(\mu):=\sum_{m=1}^2\frac{1}{2\pi}p.v.\int_{\mathbb{R}}\frac{f_1(\mu)-f_m(\rho)}{|\mu-\rho|^2+|f_1(\mu)-f_m(\rho)|^2}\omega_m(\rho)\big[\psi_{\varep_5^{10}}(\beta_m(\rho))-\psi_{\sqrt{\eps}}(\mu-\rho)\big]\,d\rho,
\end{split}
\end{equation*}
and
\begin{equation*}
\begin{split}
&E^+=E^+_1+E^+_2+E^+_3,\\
&E^+_1(\mu):=\frac{\omega_2(\mu)|f'_2(\mu)|^2}{2(1+|f'_2(\mu)|^2)},\\
&E^+_2(\mu):=-\big[\widetilde{R}^{-1}\big(S_t^\infty(z(\alpha_2+\beta_2(\mu)))\big)\big]_1,\\
&E^+_3(\mu):=\sum_{m=1}^2\frac{1}{2\pi}p.v.\int_{\mathbb{R}}\frac{f_2(\mu)-f_m(\rho)}{|\mu-\rho|^2+|f_2(\mu)-f_m(\rho)|^2}\omega_m(\rho)\big[\psi_{\varep_5^{10}}(\beta_m(\rho))-\psi_{\sqrt{\eps}}(\mu-\rho)\big]\,d\rho.
\end{split}
\end{equation*}

We have to verify the assumptions \eqref{pi25} in Proposition \ref{TechProp}. More precisely, we show that
\begin{equation}\label{las4}
\sum_{m=1}^3\big[\|E^+_m\|_Z+\|E^-_m\|_Z\big]\lesssim d\log(1/d).
\end{equation}
The bounds on $\|E^+_2\|_Z$ and $\|E^-_2\|_Z$ follow from \eqref{pi10} and \eqref{ko102}. Similarly, the bounds on $\|E^+_1\|_Z$ and $\|E^-_1\|_Z$ follow from \eqref{pi11} and the estimate
\begin{equation*}
|E^-_1(\mu)|+|E^+_1(\mu)|\leq |\omega_1(\mu)||f'_1(\mu)|^2+|\omega_2(\mu)||f'_2(\mu)|^2\lesssim\mu^2,
\end{equation*}
which is a consequence of \eqref{ko76}. Finally, to bound $\|E^+_3\|_Z$ and $\|E^-_3\|_Z$ we use again \eqref{pi10}. It suffices to prove that
\begin{equation}\label{las5}
\Big\|\int_{\mathbb{R}}\frac{f_n(\mu)-f_m(\rho)}{|\mu-\rho|^2+|f_n(\mu)-f_m(\rho)|^2}\omega_m(\rho)\big[\psi_{\varep_5^{10}}(\beta_m(\rho))-\psi_{\sqrt{\eps}}(\mu-\rho)\big]\,d\rho\Big\|_{C^1_\mu(I_{\sqrt{\eps}/8})}\lesssim\log(1/\eps),
\end{equation}
for $m,n\in\{1,2\}$. Since $|\mu|\leq\sqrt\eps/8$, we may assume that the integral in $\rho$ in \eqref{las5} is taken over $|\rho|\geq\sqrt\eps/8$. Then the bound \eqref{las5} follows easily using the estimates $|\omega_m(\rho)|\lesssim 1$, $|f_m(\rho)|\lesssim \eps+\rho^2$, $|f'_m(\rho)|\lesssim |\rho|$, $m\in\{1,2\}$. This completes the proof of \eqref{las4}.

We consider now the function $\mu\to [W_1(\mu)]_2-[W_2(\mu)]_2$ and write, using \eqref{las2} and \eqref{las2.5},
\begin{equation*}
[W_1(\mu)]_2-[W_2(\mu)]_2=[T_{4,(g,f)}\omega^+(\mu)-T_{2,f}\omega^+(\mu)]+[T_{2,g}\omega^-(\mu)-T_{4,(f,g)}\omega^-(\mu)]+E(\mu),
\end{equation*}
where
\begin{equation*}
E=E_1+E_2+E_3,
\end{equation*}
\begin{equation*}
E_1(\mu):=\frac{\omega_1(\mu)f'_1(\mu)}{2(1+|f'_1(\mu)|^2)}-\frac{\omega_2(\mu)f'_2(\mu))}{2(1+|f'_2(\mu)|^2)},
\end{equation*}
\begin{equation*}
E_2(\mu):=\big[\widetilde{R}^{-1}\big(S_t^\infty(z(\alpha_1+\beta_1(\mu)))\big)\big]_2-\big[\widetilde{R}^{-1}\big(S_t^\infty(z(\alpha_2+\beta_2(\mu)))\big)\big]_2,
\end{equation*}
and
\begin{equation*}
\begin{split}
E_3(\mu):=\sum_{m=1}^2\frac{1}{2\pi}p.v.\int_{\mathbb{R}}&\Big[\frac{\mu-\rho}{|\mu-\rho|^2+|f_1(\mu)-f_m(\rho)|^2}-\frac{\mu-\rho}{|\mu-\rho|^2+|f_2(\mu)-f_m(\rho)|^2}\Big]\\
&\times\omega_m(\rho)\big[\psi_{\varep_5^{10}}(\beta_m(\rho))-\psi_{\sqrt{\eps}}(\mu-\rho)\big]\,d\rho.
\end{split}
\end{equation*}
The bounds \eqref{ko76} show that $|E_1(\mu)|\lesssim\eps+|\mu|$. Moreover, using \eqref{ko102}, 
\begin{equation*}
|E_2(\mu)|\lesssim \big|z(\alpha_1+\beta_1(\mu))-z(\alpha_2+\beta_2(\mu))\big|\lesssim \eps+|\mu|.
\end{equation*}
Finally, if $|\mu|\leq\sqrt\eps/8$ then we may assume that the integral in $\rho$ in the definition of $E_3$ is taken over $|\rho|\geq \sqrt\eps/8$ and estimate, using again \eqref{ko76},
\begin{equation*}
|E_3(\mu)|\lesssim \int_{|\rho|\geq \sqrt\eps/8}\frac{|\rho|\cdot\rho^2\eps}{\rho^4}\psi_(\rho)\,d\rho\lesssim \eps\log(1/ \eps).
\end{equation*}
Therefore, for any $\phi\in\mathcal{C}_2$ supported in $I_{\sqrt\eps/8}$,
\begin{equation*}
\Big|\int_{\mathbb{R}}E\cdot\phi\,d\alpha\Big|\lesssim \int_{|\alpha|\leq\sqrt\eps/8}\big[\eps\log(1/\eps)+|\alpha|\big]\frac{\eps}{\eps^2+\alpha^2}\,d\alpha\lesssim \eps\log(1/\eps).
\end{equation*}
Therefore all the conditions in \eqref{pi25} are verified, and the desired conclusion follows from \eqref{pi27}.
\end{proof}

We prove now a key integral inequality.

\begin{lemma}\label{intineq}
Assume that $I\subseteq [0,T]$ is a subinterval with the property that
\begin{equation}\label{las10}
\inf_{\alpha,\beta\in\mathbb{R}}F(z)(\alpha,\beta,t)\leq\varep_4\qquad\text{ for any }t\in I,
\end{equation}
where $\varep_4$ is the constant in Proposition \ref{chord-arc3}. For any $t\in I$
\begin{equation}\label{las11}
\mathcal{D}(t):=\sup_{\alpha,\beta\in\mathbb{R},\,|\alpha-\beta|\in[\varep_1,\varep_1^{-1}]}|z(\alpha,t)-z(\beta,t)|^{-2}.
\end{equation}
Then, for any $t_1\leq t_2\in I$,
\begin{equation}\label{las12}
\mathcal{D}(t_2)-\mathcal{D}(t_1)\lesssim \int_{t_1}^{t_2}\mathcal{D}(s)\log(\mathcal{D}(s))\,ds.
\end{equation}
\end{lemma}

\begin{proof}[Proof of Lemma \ref{intineq}] Assuming that $t_1,t_2\in I$ are fixed, we apply first Proposition \ref{chord-arc3} to fix points $\alpha_1=\alpha_1(t_2)$ and $\alpha_2=\alpha_2(t_2)$ such that
\begin{equation*}
\begin{split}
&|\alpha_1-\alpha_2|\in [2\varep_1,\varep_1^{-1}/2];\\
&|z(\alpha_1,t_2)-z(\alpha_2,t_2)|=\inf_{\alpha,\beta\in\mathbb{R},\,|\alpha-\beta|\in[\varep_1,\varep_1^{-1}]}|z(\alpha,t_2)-z(\beta,t_2)|\lesssim \varep_4.
\end{split}
\end{equation*}
For $t\in I$, $t\leq t_2$, we define
\begin{equation*}
\widetilde{\mathcal{D}}(t)=\widetilde{\mathcal{D}}_{t_2}(t):=|z(\alpha_1,t)-z(\alpha_2,t)|^{-2}.
\end{equation*}
Using \eqref{ko51} and \eqref{ko17.2}, we notice that 
\begin{equation*}
\begin{split}
\widetilde{\mathcal{D}}'(t_2)&=-2\widetilde{\mathcal{D}}(t_2)^2\big[[z(\alpha_1,t_2)-z(\alpha_2,t_2)]\cdot [z_t(\alpha_1,t_2)-z_t(\alpha_2,t_2)]\big]\\
&=-2\widetilde{\mathcal{D}}(t_2)^2\big[[z(\alpha_1,t_2)-z(\alpha_2,t_2)]\cdot [v^1(\alpha_1,t_2)-v^1(\alpha_2,t_2)]\big].
\end{split}
\end{equation*}
With the notation in subsection \ref{BiRo},
\begin{equation*}
\big|[z(\alpha_1,t_2)-z(\alpha_2,t_2)]\cdot [v^1(\alpha_1,t_2)-v^1(\alpha_2,t_2)]\big|=d\big|[W_1(0)]_2-[W_2(0)]_2\big|.
\end{equation*}
Using also Lemma \ref{diffvel} it follows that
\begin{equation}\label{las15}
|\widetilde{\mathcal{D}}'(t_2)|\lesssim \widetilde{\mathcal{D}}(t_2)\log (\widetilde{\mathcal{D}}(t_2)).
\end{equation}

Recalling the qualitative assumption \eqref{qual}, there is $\delta>0$ with the property that $F(z)(\alpha,\beta,t)\geq\delta$ for all $\alpha,\beta\in\mathbb{R}$ and $t\in[0,T]$.\footnote{We note that the small constant $\delta$ is not an effective constant, depending only on $A$ and $T$.} Therefore the inequality \eqref{las15} persists in a small neighborhood of the point $t_2$, i.e. there is $\delta'>0$ such that
\begin{equation*}
|\widetilde{\mathcal{D}}'(t)|\lesssim \widetilde{\mathcal{D}}(t)\log (\widetilde{\mathcal{D}}(t))\qquad\text{ for any }t\in [t_2-\delta',t_2]\cap[0,T].
\end{equation*}
Using the simple observation $\widetilde{\mathcal{D}}(t)\leq \mathcal{D}(t)$, it follows that
\begin{equation}\label{las16}
\mathcal{D}(t_2)=\widetilde{\mathcal{D}}(t_2)\leq \widetilde{\mathcal{D}}(t)+C\int_{t}^{t_2}\widetilde{\mathcal{D}}(s)\log(\widetilde{\mathcal{D}}(s))\,ds\leq\mathcal{D}(t)+C\int_{t}^{t_2}\mathcal{D}(s)\log(\mathcal{D}(s))\,ds,
\end{equation}
for any $t\in [t_2-\delta',t_2]\cap[0,T]$, where $C$ is a constant that may depend only on $A$ and $T$. The desired inequality \eqref{las12} follows by decomposing the interval $[t_1,t_2]$ into sufficiently many small intervals, and applying \eqref{las16}.  
\end{proof}

Finally, we can complete the proof of the main theorem.

\begin{proof}[Proof of Theorem \ref{MainThm}] Assume, for contradiction, that there is $\overline{t}\in[0,T]$ such that 
\begin{equation}\label{las17}
\inf_{\alpha,\beta\in\mathbb{R}}F(z)(\alpha,\beta,\overline{t})\leq\varep_4/K,
\end{equation}
where $K$ is a sufficiently large constant and $\varep_4$ is as in Proposition \ref{chord-arc3}. By continuity, there is $t_0\in[0,\overline{t}]$ such that
\begin{equation*}
\inf_{\alpha,\beta\in\mathbb{R}}F(z)(\alpha,\beta,t_0)=\varep_4/2\quad\text{ and }\quad \inf_{\alpha,\beta\in\mathbb{R}}F(z)(\alpha,\beta,t)\leq\varep_4/2\quad\text{ for any }\quad t\in[t_0,\overline{t}].
\end{equation*}
We apply Lemma \ref{intineq} on the interval $[t_0,\overline{t}]$. The inequality \eqref{las12} shows easily that
\begin{equation*}
\mathcal{D}(\overline{t})\leq\mathcal{D}(t_0)e^{C_0e^{C_0(\overline{t}-t_0)}}
\end{equation*}
for some constant $C_0$ that may depend only on $A$ and $T$. At the same time $\mathcal{D}(t_0)\approx \varep_4^{-2}$. Therefore $\mathcal{D}(\overline{t})\lesssim\varep_4^{-2}$, which shows that
\begin{equation*}
\inf_{\alpha,\beta\in\mathbb{R}}F(z)(\alpha,\beta,\overline{t})\gtrsim\varep_4.
\end{equation*}
This is in contradiction with \eqref{las17}, which completes the proof of the theorem. 
\end{proof}

\end{document}